\documentclass[11pt]{article}

\usepackage[utf8]{inputenc}
\usepackage{amsmath,amsthm,amssymb,enumerate,framed,mdwlist}
\usepackage{color,colortab}
\usepackage[square,numbers,sort&compress]{natbib}
\usepackage{graphicx}

\newcommand{\R}{\mathbb{R}}

\renewcommand{\epsilon}{\varepsilon}

\newtheoremstyle{mythmstyle}
	{\topsep}
	{\topsep}
	{\itshape}
	{}
	{\scshape}
	{.}
	{3pt}
	{}
\theoremstyle{mythmstyle}

\newtheorem{nn}{}[section]
\newtheorem{lemma}[nn]{Lemma}
\newtheorem{theorem}[nn]{Theorem}
\newtheorem{cor}[nn]{Corollary}

\newtheorem{definition}[nn]{Definition}

\newtheorem{conj}[nn]{Conjecture}

\newtheorem{REMARK}[nn]{Remark}

\numberwithin{equation}{section}

\allowdisplaybreaks
\usepackage{authblk}
\title{On the Length of Monotone Paths in Polyhedra}
\author[1]{M. Blanchard}
\author[2]{J.A. De Loera}
\author[3]{Q. Louveaux}
\affil[1]{Massachusetts Institute of Technology, Cambridge, MA, USA}
\affil[2]{University of California, Davis, CA, USA}
\affil[3]{Université de Liège, Liège, Belgique}
\date{} 
\begin{document}

%\title{Monotone \& Simplex Paths on Combinatorial Polyhedra}

%\author{M. Blanchard, J.A. De Loera, Q. Louveaux}

\maketitle

\begin{abstract}
    Motivated by the problem of bounding the number of iterations of the Simplex algorithm 
    we investigate the possible lengths of monotone paths followed by the Simplex 
    method inside the oriented graphs of polyhedra (oriented by the objective function). We consider both the \emph{shortest} 
    and the \emph{longest} monotone paths and estimate the \emph{monotone diameter} and \emph{height} of polyhedra. Our analysis applies to  transportation polytopes, matroid polytopes, matching polytopes, shortest-path polytopes, and the TSP, among others. 
    
    We begin by showing that combinatorial cubes have monotone and Bland pivot height bounded by their dimension and that in fact  all monotone paths of zonotopes are no larger than the number of edge directions of the zonotope. We later use this to show that several polytopes have polynomial-size pivot height, for \emph{all} pivot rules. 
 In contrast, we show that many well-known combinatorial polytopes have exponentially-long monotone paths. 
 Surprisingly, for some famous pivot rules,  e.g., greatest improvement and steepest edge, these same polytopes 
 have polynomial-size simplex paths. 
\end{abstract}

\section{Introduction}
It is a famous open challenge to find a pivot rule that can make the Simplex method run in polynomial time for all LPs or
show that none exist (see e.g., \cite{miketoddsurvey2002, mysurvey2011, adleretal2014} and the many references therein for a 
discussion of this famous algorithmic problem). In particular, such a pivot rule will take polynomially many monotonically-improving
edge steps from any initial vertex. This paper discusses the possible lengths of the paths followed by the Simplex method on several famous combinatorial polyhedra where computing monotone paths has nice combinatorial meaning.

%Its one-dimensional and 
%zero-dimensional faces define a graph which is given a unique orientation using the  objective function. 
%The Simplex method follows a directed monotone path toward the optimum vertex. A \emph{pivot
%rule} decides which arc to take at each iteration. We call such paths \emph{Simplex paths}. 
%More generally 
%
%Note that all Simplex paths are monotone paths with respect to an objective function, but 
%not all monotone paths are necessarily Simplex paths because, unlike the Simplex paths, the 
%local  choice of the next arc of the path may not be consistent with one single pivoting rule.
%In this paper we try to  investigate both classes of directed paths  in linear programs. 
%We bound the length of monotone paths, showing upper bounds for the shortest 
%paths and illustrating long monotone paths. We also look at the length of Simplex paths 
%depending on the specific choice of some well-known pivot rules. 

In what follows we consider a polytope/polyhedron $P(A,b)$ in the canonical forms $\{x \in \R^n: Ax = b, x\geq 0\}$ or  $\{x \in \R^n: Ax \leq b,x\geq 0\}$. Here $A \in \R^{m\times n}$ and $b \in \R^m$.
Objective function vectors will be typically denoted by $c \in \R^n$. $LP(A,b,c)$ will denote the (minimization) LP 
instance given by $A, b, c$. 
%Given a matrix $A\in \R^{m\times n}$ with full row rank, a basis of $A$ is any invertible 
%$m\times m$ submatrix.  
Given any $A,b,c$ such that $c$ is a nondegenerate linear objective function i.e., no two vertices have the same objective function, 
one has a natural directed acyclic graph on the vertices and edges of the polytope $P(A,b)$ by orienting each edge of the polytope $P(A,b)$ 
as per the objective value of the two endpoints. This will be denoted by $G(A,b,c)$. This kind of orientations of the graphs of 
$P(A,b)$ are called \emph{LP-admissible}. We note that the directed graph $G(A,b,c)$ is always acyclic, with a unique 
sink and source in each face. While there is a characterization of the LP-admissible orientations for $3$-polytopes (see \cite{Klee-Mihalisin2000}), a similar result in higher dimensions seems unlikely (see \cite{Develin2004}).

We introduce now the main combinatorial definitions and then give several remarks about these concepts. See 
Figure \ref{klee-minty-example} for an example.

%Klee and Mihalisin \cite{Klee-Mihalisin2000} first proved the surprising fact that for any acyclic orientation of the graph of a 
% polyhedron there exist both convex and concave functions that induce that orientation of the graph. 
% But the same problem for higher dimension is essentially very difficulty. This is illustrated by 
% the work of \cite{GartnerSchurr-2006} who showed that any linear program in $n$ nonnegative variables and $m$ equality 
% constraints defines in a natural way a unique sink  orientation of the $n$-dimensional cube. From the sink of the cube, we 
% can either read off an optimal solution to the LP, or we obtain certificates for infeasibility or unboundedness, which shows 
% already the affine orientations of a cube are not easy to handle.   

\begin{definition}
Let $c$ be a linear objective function and $\pi$ a pivot rule.
\begin{enumerate}
    \item A \emph{$c$-monotone path} is a directed path in the LP-admissible oriented graph $G(A,b,c)$, that starts from some vertex to the optimal vertex (note that we always consider the optimal vertex to be the terminal node of the path, but we do not start necessarily at a specific node). 
    
    \item From each vertex there is at least one shortest $c$-monotone path to the optimum. The \emph{$c$-monotone diameter} is the maximum length of a shortest $c$-monotone path, the maximum being taken over all starting vertices.
    
    \item The \emph{$c$-height} is the length of the longest $c$-monotone path. 
    
    \item A \emph{$c$-$\pi$-simplex path} is a $c$-monotone path in $G(A,b,c)$ following the pivot rule $\pi$. In this paper we will consider four popular pivot rules: Bland's pivot rule, Dantzig's pivot rule, greatest improvement pivot rule, and steepest edge pivot rule.
\end{enumerate}
\end{definition}

We use these definitions to build our main concepts of interest.

\begin{definition}
\begin{enumerate}
    \item The \emph{monotone diameter} of a polytope is the maximum  $c$-monotone diameter, the maximum being taken over all  objective functions $c$.
    \item The \emph{height} is the maximum $c$-height, the maximum being taken over all objective functions $c$.
    \item The \emph{$\pi$-pivot height} is the maximum length of a $c$-$\pi$-simplex path for the pivot rule $\pi$, the maximum being taken over all objective functions $c$.
\end{enumerate}
\end{definition}

\begin{figure}
    \centering
    \includegraphics[width=4cm]{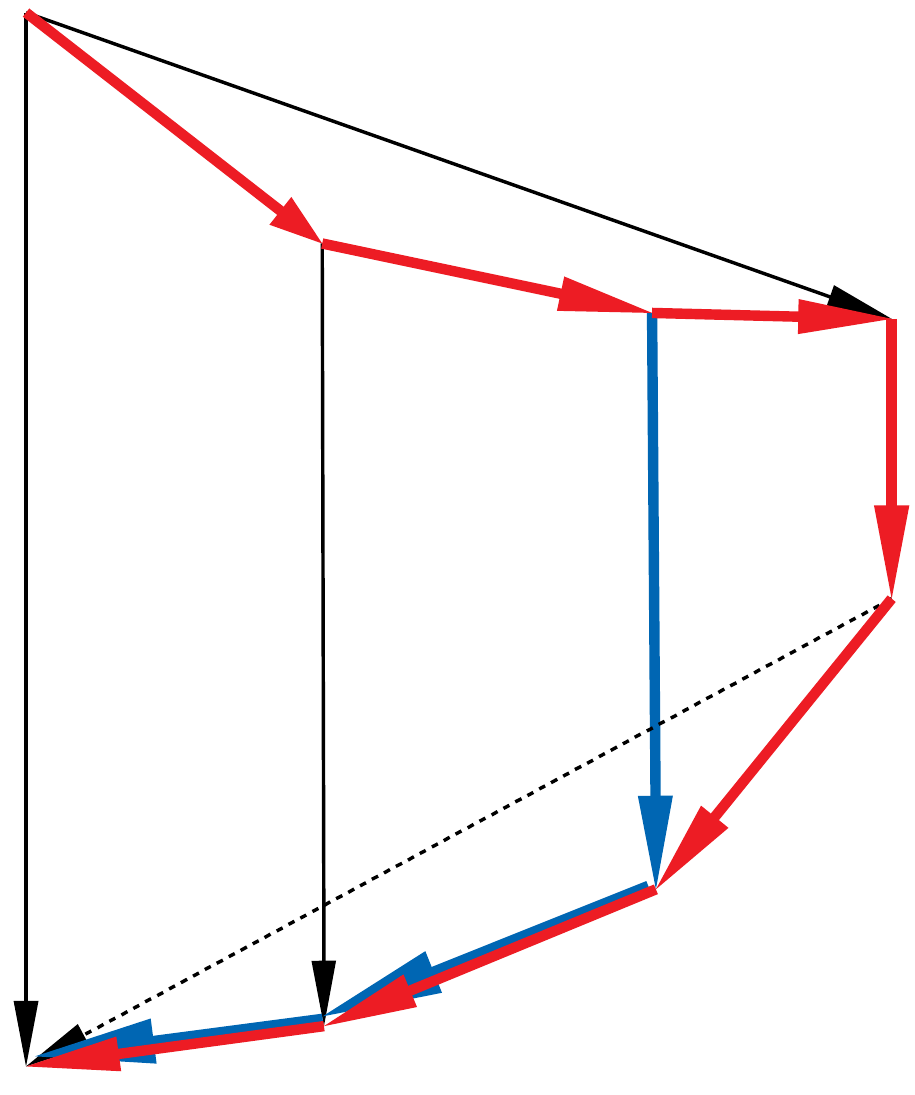}
    \caption{Two monotone paths on the directed graph $G(A,b,c)$ of the Klee$-$Minty cube. The longest monotone path 
    in red gives the \emph{height} and the blue monotone path gives the \emph{monotone diameter} of this polytope.}
\label{klee-minty-example}
\end{figure}

The study of the \emph{undirected} diameter of the graph of polytopes is of course classical and related to the Hirsch conjecture 
(see e.g., \cite{Paco-survey2013} and references), but the investigations of \emph{directed monotone} paths are even more
directly relevant to the Simplex method, and they have occupied researchers for some time:
In the 1960's Klee initiated the study of short/long monotone paths in his papers \cite{Klee1,Klee2} where he proved 
bounds on the monotone diameter and height of simple polytopes. Later in the 1980's, in a remarkable \emph{tour de force}, 
Todd \cite{monoHirschfalse} showed that the \emph{monotone} Hirsch conjecture, saying that the monotone diameter is always less or 
equal to the number of facets minus the dimension, is false.
 In the 1990's Kalai \cite{kalaiheight} proved that for a $n$-dimensional polyhedron with $m$ facets there is a subexponential bound  on the \emph{monotone} diameter of $m^{2\sqrt{m}}$. Today, several papers continue the study of shortest monotone paths (see \cite{2018-morris+gallagher} and references therein). 

It must be noted that there are several exponentially-long $c$-$\pi$-simplex paths, e.g., the Klee-Minty cubes \cite{kleeminty} 
and similar counterexamples for other pivot rules (see discussion and references in \cite{mysurvey2011}). 
The notion of height is useful to indicate the worst possible case of the Simplex method.  In fact, long monotone
paths have also been explored before, the \emph{monotone upper bound problem} asks for the maximal number $M(n, m)$ of 
vertices on a strictly increasing edge-path on a simple $n$-dimensional polytope with $m$ facets. This is the same as 
the largest height over all simple $n$-polytopes with $m$ facets. It was conjectured that 
$M(n,m)$ is never more than the number of vertices of a dual-to-cyclic $n$-polytope with $m$ facets, 
but proven to be strictly less than that in dimension six \cite{Julian+Guenter}. In our paper the reader can 
observe how the bound $M(n,m)$ is often too big for specific polytopes. In fact, Del Pia and Michini \cite{DelPia2019} recently proved that for lattice polyopes contained in $[0,k]^n$, there exists a pivot rule $\pi$ for which the $\pi$-pivot height is at most $O(n^6 k\log n)$.

We wish to stress that computational complexity influences the geometry of monotone paths of polytopes. In
\cite{adleretal2014} it was shown that there are Simplex pivoting rules for which it is PSPACE-complete to decide whether a 
particular basis will appear on the algorithm's path. This happens even for the Dantzig pivot rule \cite{fearnley+savani}. 
Moreover, it was recently shown in \cite{DKS-withcircuits} that it is hard to compute the monotone diameter. It is also difficult to decide when a monotone path is a simplex path but we focus here on four well-known pivot rules: Bland's pivot rule, 
Dantzig's pivot rule, greatest improvement pivot rule, and steepest edge pivot rule. 

Finally the concepts we discuss in this paper satisfy the following relation:
\begin{equation*}
    \text{\em (undirected) diameter} \leq \text{\em monotone diameter} \leq \pi \text{\em -pivot height} \leq \text{\em height}.
\end{equation*}

We now summarize our main results about these concepts.

\subsection*{Our results}

In Section \ref{cubes&zonotopes} we show that combinatorial cubes have monotone and Bland pivot height bounded by 
their dimension. Similarly, zonotopes have height never larger than the number of edge directions of the zonotope.

\begin{theorem}
\label{mono_diam_zonotope} Let $P$ be a convex polytope.
Denote by $Z(P)$ the zonotope generated by $E$, the minimal set of vectors containing all directions of edges of $P$.
$\text{mono-diam}(P) \leq \text{mono-diam}(Z(P)) = \text{number of different edge directions of }P$.
\end{theorem}

This simple theorem has nice consequences. We can easily show that matroid polytopes, polymatroid polytopes, and some types of transportation polytopes have  polynomial-size height and pivot height for \emph{all} pivot rules. 

\begin{cor}
\label{monodiam_matroid}
If $P$ is a matroid polytope or a polymatroid polytope, then $\text{mono-diam}(P)\leq {\binom{n}{2}}$, where $n$ is the number of
elements of the matroid.
\end{cor}

This follows immediately from the bounds on the number of edge directions of a matroid or polymatroid polytope, see Theorem 5.1. in \cite{topkis}. Therefore, polytopes such as the permutahedron or the spanning tree polytope behave particularly well in the Simplex method as all pivot rules are efficient. This is also true in other cases.
%\begin{cor}
%If $P$ is a network flow polytope coming from a graph with $E$ edges, then
% $\text{mono-diam}(P)\leq {E \choose 2}$.
%\end{cor}

\begin{theorem} \label{cor-TPsfixedk}
If $P$ is a $k\times n$ transportation polytope, 
$\text{mono-diam}(P) \leq e\cdot k!\, n^k$. Therefore, the monotone diameter of $k\times n$ transportation polytopes for fixed $k$ is polynomial in $n$.
\end{theorem}

In Section \ref{01polyhedra} we show that many well-known combinatorial polytopes have exponentially-long monotone paths, and thus exponential height.

\begin{theorem}
\label{long_path_matching}
The height of the matching, perfect matching, fractional matching and fractional perfect matching polytopes on the complete graph $K_n$ is $>C \cdot \lfloor \frac{n}{2}-1 \rfloor!$ for a universal constant $C>0$.
\end{theorem}

\begin{theorem}
\label{long_path_tsp}
There exist monotone paths of length $>C \cdot \phi^n$ on the perfect 2-matching polytope and on the TSP with $n$ nodes for a universal constant $C>0$ and $\phi=\frac{1+\sqrt 5}{2}$ the golden ratio.
\end{theorem}

\begin{theorem}
\label{long_path_shortestpathpolytope}
There exist monotone paths of length $>\frac{C}{n^2}\sqrt[3]{n!}$ on the shortest path polytope on the complete graph $K_n$ for some universal constant $C>0$.
\end{theorem}

In contrast, we prove that, for at least one of three famous pivot rules, Bland's, greatest improvement and steepest edge, they have polynomial-size pivot height.  Our discussion includes matching polytopes, fractional matching polytopes, shortest-path polytopes, and the TSP.

\begin{theorem}
\label{pivot_diameter_bland}
The number of vertices visited, by Dantzig or greatest improvement pivot rules paths, is bounded by
\begin{enumerate}[a)]
\item $m\left[ n\log (2n) \right]$  for the fractional perfect matching polytope ($FPM$) on a graph with $n$ nodes and $m$ edges. For the complete graph $K_n$, we get a bound $\sim \frac{n^3}{2}\log n$.
\item $m\left[ 2n\log (2n) \right]$  for the fractional matching polytope ($FM$) on a graph with $n$ nodes and $m$ edges. For the complete graph $K_n$, we get a bound $\sim n^3\log n$.
%\item $\frac{n(n-1)}{2}\left[ n\log n \right] \sim \frac{n^3}{2}\log n$ for the perfect 2-matching polytope on $K_n$.
\item $n^2\left[ n\log (2n-1) \right] \sim n^3\log n$ for the Birkhoff polytope on the bipartite graph $K_{n,n}$.
\item $(n^2-2n+1)\left[(n-1)\log (n-1) \right] \sim n^3\log n$ for the shortest path polytope on $n$ nodes.
%%%\item $\frac{n(n-1)}{2}\left[ n^2\log 2 \right] \sim n^4\log n$ for the undirected traveling salesman polytope on $K_n$.
\end{enumerate}
\end{theorem}

\begin{theorem}
\label{pivot_diameter_steepest}
For the steepest-edge pivot rule paths, the number of visited vertices is bounded by
\begin{enumerate}[a)]
\item $m\left[ 2n\sqrt {n}\log (2n) \right]$  for the fractional perfect matching polytope ($FPM$) on a graph with $n$ nodes and $m$ edges. For the complete graph $K_n$, we get a bound $\sim n^3\sqrt{n} \log n$.
\item $m\left[ 4n\sqrt {2n}\log (2n) \right]$  for the fractional matching polytope ($FM$) on a graph with $n$ nodes and $m$ edges. For the complete graph $K_n$, we get a bound $\sim 2n^3\sqrt{2n} \log n$.
%\item $\frac{n(n-1)}{2}\left[ \frac{n\sqrt n}{2}\log n \right] \sim \frac{n^3\sqrt n}{4}\log n$ for the perfect 2-matching polytope
\item $n^2\left[ n\sqrt{\frac{n}{2}}\log (2n-1) \right] \sim n^3\sqrt{\frac{n}{2}}\log n$ for the Birkhoff polytope.
\item $(n^2-2n+1)\left[(n-1)\sqrt{\frac{2n}{3}}\log (n-1) \right] \sim n^3\sqrt {\frac{2n}{3}}\log n$ for the shortest path polytope.
%%%\item $\frac{n(n-1)}{2}\left[ n^2\log 2 \right] \sim n^4\log n$ for the undirected traveling salesman polytope.
\end{enumerate}
\end{theorem}

In Section \ref{transportation_polytopes} we consider again the problem of estimating the monotone diameter of transportation polytopes.

\begin{theorem}
\label{monodiam_tp_polytope}
A $2 \times n$ transportation polytope has monotone diameter $\leq n$. Therefore, $2\times n$ transportation problems satisfy the monotone Hirsch conjecture.
\end{theorem}

\section{Monotone and Simplex paths on Cubes \& Zonotopes}
\label{cubes&zonotopes}

In this section we present several results about monotone paths and simplex paths on cubes and zonotopes. We will see they
have more general applicability.

\begin{theorem}
\label{monodiam_cube}
Let $C\subset \mathbb R^n$ be a polytope combinatorially equivalent to a hypercube. 
Then $\text{mono-diam}(C) = n$. Furthermore, there exists an ordering of the facets of $C$ such that Bland's pivot rule reaches the optimal solution in at most $n$ steps for any initial vertex.
\end{theorem}
\begin{proof}
We first prove by induction on $n$ that the monotone diameter is $n$. For $n=1$, the result is trivial.
Assume now that the result is true for any combinatorial cube up to dimension $n-1$.
Consider an arbitrary vertex $x\neq x^*$ of $C$. 
There must exist an improving edge going out of $x$. Consider a facet $F$ containing $x$ but that does not contain this edge.
If $x^*\in F$, we are done by the induction hypothesis. 
Otherwise $x^*$ belongs to the ``opposite'' facet. This is the facet of the polytope which does not have any vertex in common with the facet $F$. For example if $C$ is the regular hypercube, this is the parallel facet to $F$. We take the improving edge to that opposite facet and apply the induction hypothesis. To conclude note that the monotone diameter is exactly $n$ because there exists a vertex which needs at least $n$ pivots to reach the optimal solution.
\medskip

Now let us present good orderings of the facets for Bland's pivot rule.
Denote by $x^*$ the optimum vertex. We choose an ordering such that the first $n$ facets satisfy $x^* \notin F_i$ for $1\leq i\leq n$ and the last $n$ facets satisfy $x^*\in F_i$ for $n+1\leq i\leq 2n$. We will prove that Bland's rule with this ordering follows the path described above. More precisely, we prove that at each step, the index of the entering variable is in $\{1,\ldots,n\}$ while the index of the leaving variable is in $\{n+1,\ldots,2n\}$ so inserted variables will never be removed from the basis.

Consider an arbitrary vertex $x\neq x^*$. Let $i_1,\ldots,i_k>n$ be the indices such that $x\in F_i$. Since $x$ is not the optimum, there must exist an improving edge from $x$ in the cube $F_{i_1}\cap \ldots \cap F_{i_k}$ of smaller dimension $n-k$. Consider the facet $F_i$ of the $n$-dimensional cube such that $x\in F_i$ and that does not contain this improving edge. Note that $i\leq n$. Otherwise $x^*\in F_i$, therefore $F_i$ is one of the facets $F_{i_1},\ldots,F_{i_k}$ which is impossible because the improving edge is contained in their intersection.

The entering variable $\hat i$ chosen by Bland's pivot rule satisfies $\hat i\leq i\leq n$. Note that $x^*\notin F_{\hat i}$ so $x^*$ is contained in the ``opposite'' facet which corresponds to the leaving variable. Therefore the index of the leaving variable is greater than $n$. Then, variables of index $\leq n$ cannot be removed from the basis.
\end{proof}
%We prove this result by induction on $d$. For $d=2$ the result is trivial.\\
%We now assume that this holds for any polytope combinatorially equivalent to a hypercube of dimension $d-1$. Let $x$ be the initial vertex. Assuming that it is not the optimal vertex $x^*$ take an improving edge $e$ going out of $x$. Let $F$ be a facet that does not contain this edge but contains $x$.\\
%If $x^* \in F$, by induction we have an ordering of the $(d-2)$-facets of $F$ such that the Bland's rule makes at most $d-1$ steps. We construct $(d-1)$-facets of $C$ from these $(d-2)$-facets using the combinatorial direction given by $e$. For the ordering of the facets of $C$ we keep the sub-order given by this contruction. To ensure that the path stays in $F$ we can set its index to $2d$ (the maximal index) in the ordering of the facets.\\
%Otherwise, denoting by $F'$ the parallel facet to $F$, by induction we get a good ordering of the $(d-2)$-facets of $F'$. With a similar construction to the previous case, we get an ordering of the corresponding $(d-1)$-facets of $C$. We  ensure that the path first takes improving edge $e$ by setting the index of facet $F$ to 1 (minimum index) and the index of parallel facet $F'$ to $2d$ (the maximal index).

Note that these good orderings of the facets for Bland's pivot rule are very rare. Since the $n$ facets not containing the optimum should have the first $n$ indices in the ordering while the $n$ facets containing the optimum should have the last $n$ indices, there are $(n!)^2$ such good orderings among the $(2n)!$ possible orderings of the facets. Therefore, the proportion of good orderings for Bland's pivot rule is $\frac{1}{\binom{2n}{n}}$.

Next, we discuss the monotone diameter of another family of polytopes: zonotopes. A zonotope is the Minkowski sum of a set of line segments.

%{\bf WHY????} 

%{\bf DID YOU ANSWER QUENTIN's question on order?}

\begin{lemma}
Let $Z(v^1,\ldots, v^m)\subset \mathbb R^n$ be the zonotope generated by those directions. Assume any two directions $v^i,v^j$ are non-colinear. In at most $m$ augmenting steps, one can go from $\hat x$, an initial vertex, to $x^*$ the optimum. This bound is tight. Furthermore, any monotone path has at most $m$ steps so any pivot rule will take at most $m$ steps to the optimum.
\end{lemma}
\begin{proof}
Let $c\in \mathbb R^n$. We define $J^+=\{j\mid c^Tv^j>0\}$ and $J^-=\{j\mid c^Tv^j <0\}.$
Observe that $x^* =\sum_{j\in J^-} v^j$.
Consider a starting point $\hat x$. We can write it as 
$\hat x=\sum_{j\in S(\hat x)} v^j$ for a certain subset $S(\hat x)\subset \{1,\ldots,m\}$.

Two adjacent vertices of the zonotope differ by $\pm v^i$ for some $i$. Then, the only edges a monotone path can use are $\epsilon_i v^i$ where $\epsilon_i=1$ if $i\in J^-$ and $\epsilon_i=-1$ if $i\in J^+$. Furthermore, the path can follow each of these directions at most once because once $\epsilon_i v^i$ is added, this term cannot be removed by the other possible directions. Then, the length of the path is at most $m$.

Furthermore, since the admissible edges are of the form $\epsilon_i v^i$, the point $\tilde x=\sum_{j\in J^+} v^j$ is at distance at least $m$ from the optimum. Note that $\tilde x$ is a true vertex of the zonotope because it is the optimum for the cost function $-c$.
\end{proof}

%%If $\hat x\neq x^*$, there exists an improving neighbour $\tilde x$.
%$\tilde = \hat x \pm v^i$ for some $i$.
%In particular, if $c^T v^i>0$, then $\tilde x=\hat x+v^i$ and if $c^Tv^i<0$ then $\tilde x = \hat x-v^i$.
%Observe that if $i\in J^+$, that edge direction will never be removed again because we improve the objective value at each step.
%and similarly if $i\in J^-$.  Therefore there at most $m$ augmenting steps needed to reach $x^*$.

\begin{lemma}
Let $Z(v^1,\ldots, v^m)\subset \mathbb R^n$ be a zonotope. Assume any pair of directions $v^i,v^j$ are non-colinear.
It has at least $2m$ facets.
\end{lemma} 
\begin{proof}
Let $v^{i_1}, \ldots, v^{i_{d-1}}$ be a linearly independent subset of size $n-1$.
Then $Z(v^1,\ldots, v^m)$ has two facets that are translates of $Z(v^{i_1}, \ldots, v^{i_{d-1}})$.
This is because for an objective function $c\in \ker([v^{i_1}, \ldots, v^{i_{d-1}}])$ the optimum facet of $Z(v^1,\ldots, v^m)$ with respect to $\pm c$ are precisely these facets. It suffices to show that there are $\geq m$ distinct subsets of this type.

Without loss of generality, let $v^1,\ldots, v^n$ be linearly independent.\\
$S^i=\{v^1,\ldots, v^{i-1},v^{i+1},\ldots,v^n\}$, $i=1,\ldots, n$ are $n$ such subsets.
For any $v^j$ with $j\geq n+1$, there exists $v\in \{v^1, \ldots, v^n\}$
such that $\{v^1,\ldots, v^n\}\setminus \{v\} \cup \{v^j\}$ is linearly independent (by the matroid axiom).
Therefore drop $v$ and add $v^j$, the corresponding subset gives two more facets.
We get $m-n$ additional subsets like this.
\end{proof}

The following theorem shows that zonotopes satisfy the monotone Hirsch conjecture.

\begin{theorem}
$\text{mono-diam}(Z(v^1,\ldots,v^m)) = m \leq \frac{| \text{facets}|}{2} \leq |\text{facets}| - n $.
\end{theorem}

\begin{figure}
    \centering
    \includegraphics[width=12cm]{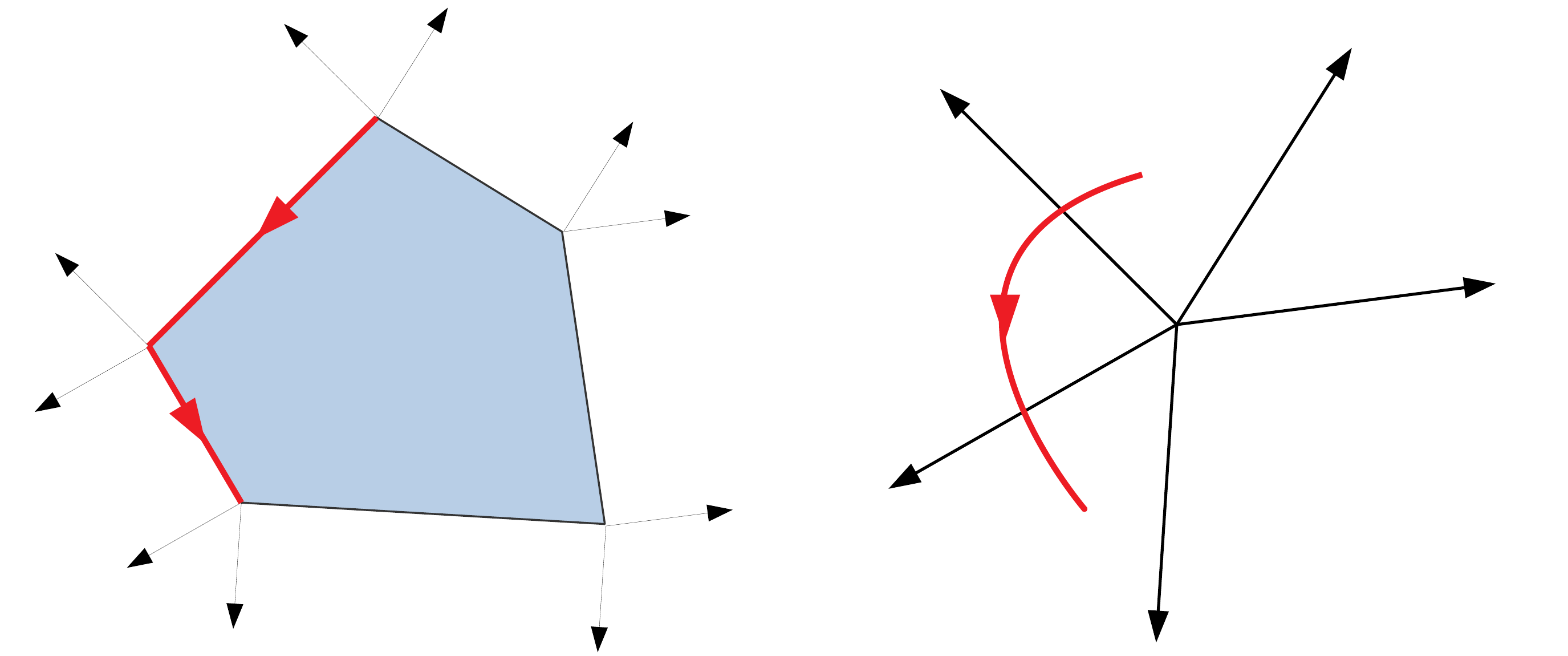}
    \caption{Left, a path on a polytope and right, the corresponding path on the normal fan.}
    \label{normal_fan}
\end{figure}

We will now use this result to bound the monotone diameter of general polytopes. For this, let us define the \emph{normal cone} of a vertex $v$ as the set of objective vectors $c$ such that $v$ is the optimum vertex for the corresponding objective function, and the \emph{normal fan} as the collection of normal cones for all vertices of the polytope (see \cite{Bonifas2014} and Figure \ref{normal_fan} for an illustration).

\begin{lemma}(Gritzmann-Sturmfels, Proposition 2.1.8. in \cite{Gritzmann_Sturmfels1993})
\label{gritzmann}
Let $P\subset \mathbb R^n$ be a polytope and let $E$ be a finite set of vectors containing all the direction of edges of $P$.
of $P$. The normal fan of the zonotope generated by $E$ is a refinement of the normal fan of $P$, therefore the length of the shortest
paths in $Z(E)$ upper bound the length of the shortest paths in $P$.
\end{lemma}

%In the following, we will denote by $Z(P)$ the zonotope $Z(E)$ where $E$ is a minimal set of vectors containing all directions of edges of $P$.

%\begin{lemma}
%\label{mono_diam_zonotope}
%$\text{mono-diam}(P) \leq \text{mono-diam}(Z(P)) = \text{number of different edge directions}$.
%\end{lemma}
\begin{proof}[Proof of Theorem \ref{mono_diam_zonotope}]
The first inequality of the theorem follows from Lemma \ref{gritzmann} because we view a path on the graph 
of the polytope as a sequence of normal fans where consecutive normal fans share a facet. The normal to this shared facet is the direction of the corresponding edge between the two vertices on the graph of the polytope. Therefore
any monotone path $p$ on the zonotope for the linear function $c$ leads to a path $\tilde p$ with smaller length on the original polytope. $\tilde p$ is still monotone for $c$ because the directions of the edges of $\tilde p$ are contained in the directions of the edges of $p$ according to Lemma \ref{gritzmann}.
\end{proof}

We can now apply Lemma \ref{mono_diam_zonotope} to several polytopes. The essential message is that if the set of edge-directions is ``small'' or polynomially bounded, then we can obtain a bound on the height using the above result. While we show
some nice situations below, in most cases this is not useful (see \cite{Onn+Pinchasi} where a lower bound on the number of edge directions is discussed).

%\begin{cor}
%If $P$ is a matroid polytope or a polymatroid polytope, then $\text{mono-diam}(P)\leq {\binom{n}{2}}$.
%\end{cor}
%\begin{proof} This follows from the bound on the number of edge directions of a matroid or polymatroid polytope, see Theorem 5.1. in \cite{topkis}.
%\end{proof}

%Therefore, polytopes such as the permutahedron or the spanning tree polytope behave particularly well in the Simplex method as
%all pivot rules are efficient. Next, we observe this is also true in other cases.
%\begin{cor}

%If $P$ is a network flow polytope coming from a graph with $E$ edges, then
% $\text{mono-diam}(P)\leq {E \choose 2}$.
%\end{cor}

%\begin{cor} %\label{cor-TPsfixedk}
%If $P$ is a $k\times n$ transportation polytope, 
%$\text{mono-diam}(P) \leq e\cdot k!\, n^k$. Therefore, the monotone diameter of $k\times n$ transportation polytopes for fixed $k$ is polynomial in $n$.
%\end{cor}

\begin{proof}[Proof of Theorem \ref{cor-TPsfixedk}]
Edges on transportation polytopes are alternating sign cycles on the bipartite graph. Since there are $k$ supply nodes, the length of the cycle is $2p$ for $2\leq p\leq k$. The number of such cycles of length $2p$ is $\frac{1}{p} \frac{n!k!}{(n-p)!(k-p)!}$. Then, the number of different edge directions is bounded by
\begin{equation*}
\sum_{p=2}^k \frac{1}{p} \frac{n^p k!}{(k-p)!} \leq n^k k! \sum_{p=2}^k \frac{1}{(k-p)! n^{k-p}} \leq e^{1/n} n^k k! 
\end{equation*}
and the proof follows.
\end{proof}

Finally, the above family in Corollary \ref{cor-TPsfixedk} is naturally generalized by $N$-fold linear programs, see Chapter 4 in \cite{DHKbook}.
In that case the defining matrix $A$ has a very specific shape as multiple copies of smaller matrices. We omit
details.

\section{Monotone and Simplex paths on $0/1$ and $0/\frac{1}{2}/1$ polyhedra}
\label{01polyhedra}

In this section, we give results on the height, monotone diameter, and the pivot height of some well-known polytopes.
It should be noted that the papers \cite{KitaharaMizuno2013,DKS-withcircuits} provide general polynomial bounds for $0/1$-polytopes and their monotone diameters are also polynomial. But in this section we look at specific families and thus we can obtain more precise polynomial bounds.

The height of a polytope $P$ is the length of the longest path in the directed graph of $P$. To start, note 
that the height gives a bound on the number of steps of the worst possible pivot rule of the Simplex algorithm. 
As shown in Section \ref{cubes&zonotopes} for some particular polytopes (e.g., zonotopes) the height is polynomially bounded, thus it gives a polynomial bound for the Simplex algorithm for \emph{any} pivot rule. However, it turns out that, for many polytopes of interest and for some well-known $0/1$ and $0/\frac{1}{2}/1$ polyhedra, monotone paths can be very long.  Here we collect some of the results about the height of polyhedra:

First, Pak \cite{pak2000} showed that the height of the Birkhoff polytope is exponential.

\begin{theorem}[Pak, Theorem 1.4. in \cite{pak2000}]
\label{pak_theorem}
There exists a linear function $\phi$ with a decreasing sequence of vertices of length $>C\cdot n!$ of the Birkhoff polytope on the bipartite graph $K_{n,n}$ for a universal constant $C>0$.
\end{theorem}

Note that the graph of a face $F$ of a polytope $P$ is a proper subgraph of the graph of $P$. Therefore the height of a polytope is greater or equal to the height of any of its faces. Indeed, let $c$ be a cost function in $F$. For $P$ take that same cost function parallel to $F$ and denote it by $\tilde c$. Then any monotone path in $(F,c)$ is a monotone path in $(P,\tilde c)$.\medskip

Let us recall the definitions and basic properties of the combinatorial polytopes we will consider in this Section. A \emph{matching} 
in a graph $G=(V,E)$ is a subset of edges $M\subset E$ such that every vertex meets at most one edge of $M$. A matching is \emph{perfect} if every vertex meets exactly one edge of $M$.
The perfect matching polytope (M) of $G$ is defined as the convex hull of the $0/1$ incidence vectors of matchings i.e.,

\begin{equation*}
    \text{M(G)} = \text{conv}\{ \chi^M: M \text{ is a matching of G}\}.
\end{equation*}

The perfect matching polytope (PM) of $G$ is the convex hull of the perfect matchings. Note that the perfect matching polytope on the complete bipartite graph is the Birkhoff polytope.

\begin{equation*}
    \text{PM(G)} = \text{conv}\{ \chi^M:  M \text{ is a perfect matching of G} \}
\end{equation*}

For these two polytopes, two matchings are adjacent if and only if the union of their support graph contains a unique cycle (see Lemma 1 in \cite{rispoli1992}). A set of inequalities describing these polytopes is given by the Edmond's matching theorem \cite{edmonds1965}.

We also consider the relaxations of these polytopes obtained by omitting the \emph{odd cycle inequalities}. The fractional matching polytope (FM) of $G$ is defined by

\begin{equation*}
        \text{FM(G)} = \{ x\in \R^E(G): \; x_e\geq 0 \; \forall e\in E(G),\; x(\delta(v))\leq 1 \; \forall v\in V(G)\}
\end{equation*}

where $E(G), V(G)$ of course denote, respectively, the sets of edges and vertices of the graph $G$. Similarly, the fractional perfect matching (FPM) is described by

\begin{equation*}
        \text{FPM(G)} = \{ x\in \R^E(G): \; x_e\geq 0 \; \forall e\in E(G),\; x(\delta(v))= 1 \; \forall v\in V(G)\}.
\end{equation*}
Note that M(G) and PM(G) are respectively a face of FM(G) and FPM(G). The adjacency of these fractional polytopes is given in Theorem 25 of \cite{behrend}. In the following we will only use the fact that the graph of M(G) and PM(G) are, respectively, a subgraph of FM(G) and FPM(G).

A 2-matching of $G$ is a subset of edges $M$ such that every vertex is incident to exactly $2$ edges in $M$. Note that a 2-matching 
is the union of disjoint cycles. The perfect 2-matching polytope (P2M) of $G$ is defined as a $0/1$ polytope as follows,

\begin{equation*}
    \text{P2M(G)} = \text{conv}\{ \chi^M: M \text{ is a perfect 2-matching of G}\}.
\end{equation*}

Two 2-perfect matchings are adjacent if and only if the symmetric difference of their support graphs contains a unique alternating cycle (see Lemma 1 in \cite{rispoli1994} and Figure \ref{adjacency_p2m} for an illustration).

In the following, if the graph is not specified we will consider the complete graph $K_n$. 

\begin{figure}
    \centering
    \includegraphics[width=12cm]{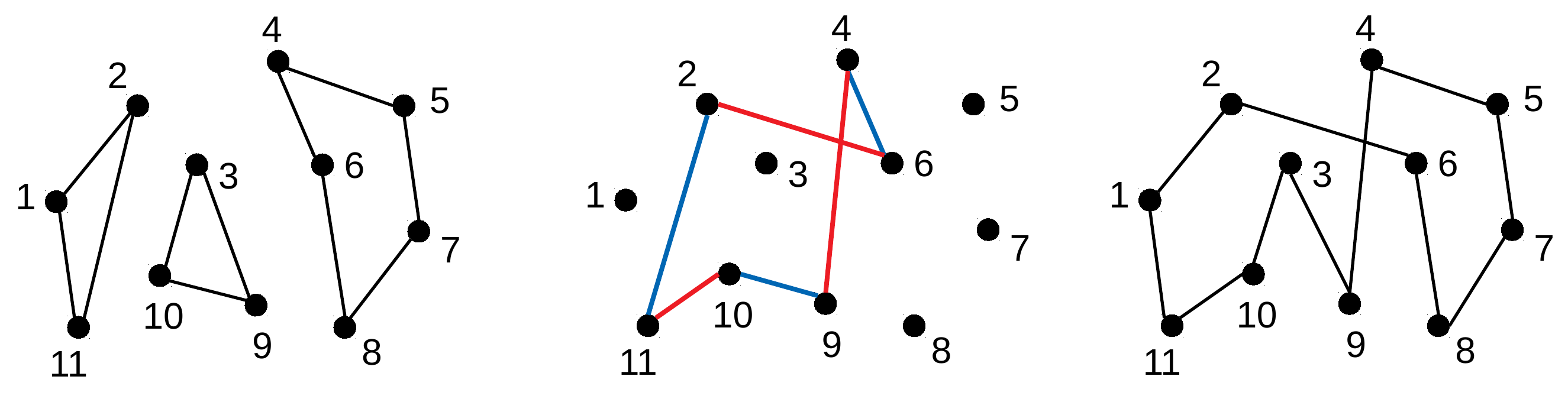}
    \caption{Left and right, two adjacent perfect 2-matchings. In the middle, the corresponding alternating cycle.}
    \label{adjacency_p2m}
\end{figure}

The traveling salesman polytope (TSP) on $K_n$ is the convex hull of tours i.e., cycles of length $n$. The TSP graph is therefore a proper subgraph of the perfect 2-matching polytope of $K_n$ (see \cite{RispoliCosares1998}).

Finally, the shortest path polytope on $n$ nodes is defined as the convex hull of paths from say node $1$ to node $n$ without cycles. A system of equations is given by
\begin{equation*}
\left \{ (x_{i,j}\geq 0)_{1\leq i\leq n-1, \ 2\leq j\leq n} : \, \sum_{j=2}^n x_{1,j} =1, \sum_{j\neq i} x_{i,j} -\sum_{j\neq i} x_{j,i}=0, \sum_{j\neq i}x_{i,j}\leq 1 ,\,2\leq i\leq n-1\right \}.
\end{equation*}
Two paths are adjacent if and only if the union of their support graph contains a unique cycle (see Lemma 2 in \cite{rispoli1992}).

%\begin{cor}
%The height of the matching, perfect matching, fractional matching and fractional perfect matching polytopes on the complete graph $K_n$ is $>C \cdot \lfloor \frac{n}{2}-1 \rfloor!$ for a universal constant $C>0$.
%\end{cor}

\begin{proof}[Proof of Theorem \ref{long_path_matching}]
We show that the Birkhoff polytope is a face of each of the considered polytopes. We will denote by $x_{i,j}$ the component corresponding to the edge between nodes $i$ and $j$ for a vertex $x$ in the polytope. Note that graphs are non-oriented here.

Define $E_1:=\{1,\ldots,\lfloor\frac{n}{2}\rfloor\}$ and $E_2:=\{\lfloor\frac{n}{2}\rfloor+1,\ldots,2\lfloor\frac{n}{2}\rfloor\}$. For the matching polytope and the fractional matching polytope, the corresponding face can be described by the several equalities $x_{i,j}=0$ for $(i,j)\notin E_1\times E_2\cup E_2\times E_1$ and $x(\delta(i))=1$ for $i\in E_1\cup E_2$. In both the matching and fractional matching polytopes, these equalities describe the set of perfect matchings on the bipartite graph between $E_1$ and $E_2$ i.e., vertices of these facets are in exact correspondence with the vertices of the $K_{E_1,E_2}$ Birkhoff polytope. Furthermore, the adjacency between the perfect matchings of these faces is exactly the same as in the Birkhoff polytope. Hence the corresponding face is equivalent to the Birkhoff polytope with $2\times \lfloor\frac{n}{2}\rfloor $ nodes. The monotone diameter of these polytopes is therefore greater than the bound $C\lfloor\frac{n}{2}\rfloor!$ given in Theorem \ref{pak_theorem}.

For the perfect matching polytope we can simply take the equalities $x_{i,j}=0$ for $(i,j)\notin E_1\times E_2\cup E_2\times E_1$. The other equalities of the form $x(\delta(i))$ are already satisfied. We get the same bound $C\lfloor\frac{n}{2}\rfloor!$ for the monotone diameter.

The same argument holds for the fractional perfect matching polytope when $n$ is even. However, when $n$ is odd, matchings on $K_{E_1,E_2}$ are not vertices of the polytope anymore. In this case, we can restrict to the face $x_{n-2,n-1}=x_{n-1,n}=x_{n-2,n}=1/2$ and use the same arguments as above with $E_1:=\{1,\ldots, \frac{n-3}{2}\}$ and $E_2:=\{\frac{n-1}{2},\ldots,n-3\}$. We finally get the bound $C\lfloor\frac{n}{2}-1\rfloor!$ for the monotone diameter.
\end{proof}

%{\bf DETAILS ARE VERY SPARSE, WE SHOULD EXPLAIN EACH CASE SEPARATELY.}

%\begin{theorem}
%\label{long_path_tsp}
%There exist monotone paths of length $>C \cdot \phi^n$ on the perfect 2-matching polytope and on the TSP with $n$ nodes for a universal constant $C>0$ and $\phi=\frac{1+\sqrt 5}{2}$ the golden ratio.
%\end{theorem}

\begin{proof}[Proof of Theorem \ref{long_path_tsp}]
Recall that $n$-tours, which are the vertices of the TSP, are also vertices of the perfect 2-matching polytope. If two tours are adjacent on the perfect 2-matching polytope, then they are also adjacent in TSP (see \cite{RispoliCosares1998}). Therefore it suffices to prove that there exists a long monotone path on the perfect 2-matching polytope going through $n$-tours only.

Denote by $x_{i,j}$ the component corresponding to the edge between nodes $i$ and $j$ for a vertex $x$ in the polytope. Consider the following linear function:
\begin{equation*}
\psi=x_{1,2}+\alpha x_{1,3}+\ldots+\alpha^{n-2} x_{1,n} + \alpha^{n-1} x_{2,3} +\alpha^n x_{2,4} +\ldots+\alpha^{\frac{n(n-1)}{2}-1} x_{n-1,n}
\end{equation*}
for $0<\alpha<1/2$ such that the linear order on the perfect 2-matching polytope or on the TSP is the lexicographic order on the edges with the following order:

$\{1,2\},\{1,3\},\ldots,\{1,n\},\{2,3\},\ldots,\{n-1,n\}$.

Denote by $x^*=(1,n,2,n-1,3,n-2,\cdots)$ the optimum for TSP, see Figure \ref{long_path_TSP}e). The initial tour is going to be the cycle $x^0=(1,2,\ldots,n)$. We will construct by induction a monotone path with exponential length. We denote by $L_n$ the length of this monotone path. For $n\geq 4$, assume that we have constructed these long monotone paths for $k=4,\cdots, n-1$. Let us now construct the path of length $L_n$.\\

Step 1:
We first restrict to $x_{1,2}=1$. We can get to the optimum $x^1$ of this facet in at least $L_{n-1}$ steps. Indeed if $x$ is a $(n-1)$-tour in the long path for $n-1$ nodes, define $\tilde x$ a  $n$-tour by dividing node $1$ into two nodes $1$ and $2$. The indices of the other nodes should be shifted by one accordingly. Recall that two 2-matchings are adjacent if and only if the symmetric difference of their edges defines a unique alternating cycle. Let $x_1$ and $x_2$ be two adjacent tours in the $(n-1)$-perfect 2-matching. Then either $\tilde x_1$ and $\tilde x_2$ are adjacent in the $n$-perfect 2-matching or $\tilde x_1$ and $\hat x_2$ are adjacent where $\hat x_2$ is the same tour as $\tilde x_2$ except the two nodes coming from the division of node $1$ have been switched. We can therefore construct a path of length $L_{n-1}$ corresponding to the same path for $(n-1)$-tours. We then get from the corresponding end point to the optimum of the facet $x_{1,2}=1$ (see Figure \ref{long_path_TSP}a)). These two tours might be distinct if we have to switch the two nodes coming from the division of node $1$, which takes at most one step.

\begin{figure}
    \centering
    \includegraphics[width=12cm]{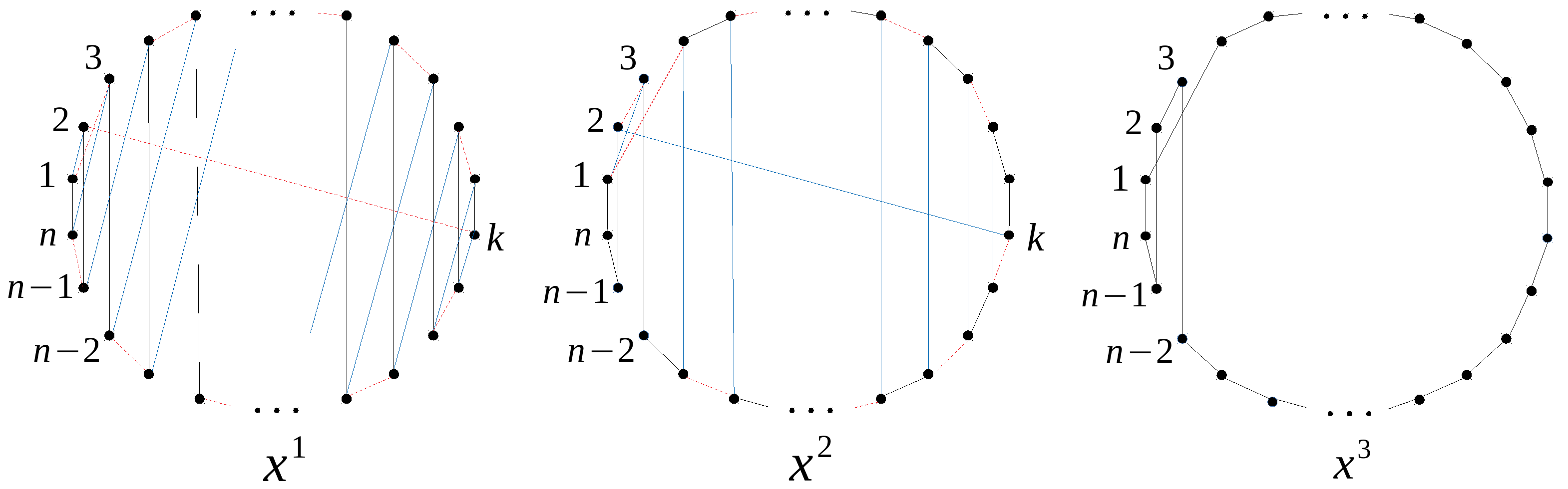}
    \caption{Step 2 of the monotone path on TSP. Edges in blue are the edges going to be deleted and the dashed red edges are going to be inserted. Since they form an alternating cycle, these tours are adjacent.}
    \label{step2}
\end{figure}

Step 2:
We now get in two improving steps to the tour $x^3=(1,4,5,\ldots,n-2,3,2,n-1,n)$ (see Figure \ref{step2}). The edges of the current vertex $x^1$ are $x^1_{i,n+1-i}=1$ for all $i$, $x^1_{1,2}=x^1_{i,n+3-i}=1$ for $i\geq 3$. We now get to the tour $x^2=(2,n-1,n,1,3,n-2,n-3,4,5,\ldots,k)$ which uses all the edges of the form $x_{i,n+1-i}$. Here $k=\frac{n}{2}+1$ if $n\equiv 0\mod 4$, $k=\frac{n}{2}$ if $n\equiv 2\mod 4$ and $k=\frac{n+1}{2}$ otherwise. This is an adjacent node because the symmetric difference of the graphs of the two tours has a unique alternating cycle $(2,1,3,n,n-1,4,5,n-2,n-3,6,7,\ldots k)$. The precise end of this alternating cycle depends on $n\mod 4$. If $n\equiv 0\mod 4$, the ending of this cycle is  $(\ldots,k-2,k-1,k+2,k+1,k)$. If $n\equiv 2\mod 4$, it is $(\ldots,k-2,k-1,k+4,k+3,k)$. For $n\equiv 1\mod 4$, it is $(\ldots,k-2,k+4,k+3,k-1,k,k+2,k+1,k)$ and for $n\equiv3\mod 4$ it is $(\ldots,k+2,k+1,k+3,k+2,k)$. Because $x^2_{1,2}=0$, this is an improving step for the lexicographic order on the edges. Now use the alternating cycle $(2,3,1,4,n-3,n-4,5,6,n-5,n-6,\ldots,k)$ to get to the neighbor tour $x^3=(1,4,5,\ldots,n-2,3,2,n-1,n)$. More precisely, the ending of the alternating cycle is $(\ldots,k-3,k-2,k+1,k)$ if $n\equiv 0\mod 4$, $(\ldots,k-2,k+3,k+2,k-1,k)$ if $n\equiv 2\mod 4$, $(\ldots,k-2,k-1,k+1,k)$ if $n\equiv1\mod 4$ and $(\ldots,k-2,k+2,k+1,k-1,k)$ if $n\equiv3\mod 4$. This is also an improving step because $x_{1,2}^3=x_{1,3}^3=0$.

Now we fix $x_{n-1,2}=x_{2,3}=x_{3,n-2}=1$. We get to optimal tour of this facet (see Figure \ref{long_path_TSP}b)) in at least $L_{n-3}$ steps, similarly to the technique used for the $L_{n-1}$ long step: in the $n$-tour, nodes $n-1,2,3$ and $n-2$ will be merged together to obtain a $(n-3)$-tour.

Step 3:
Note that now $\{1,n\}$ and $\{1,n-1\}$ are edges that will never be removed so we are restricted to the facet $x_{1,n}=x_{1,n-1}=1$. Merging together nodes $1,n,n-1$, the resulting $(n-2)$-tour is exactly the tour given at the end of step 1 for $n-2$ nodes. Apply Step 2 again to get to the next tour in at least $2+L_{n-5}$ steps which is the optimum of the facet $x_{1,n}=x_{1,n-1}=x_{n-2,3}=x_{3,4}=x_{4,n-3}=1$ (see Figure \ref{long_path_TSP}c)).
Now, $\{2,n\}$ and $\{2,n-2\}$ are edges that will never be removed so we are restricted to the facet $x_{2,n}=x_{2,n-2}=1$. With the same arguments, we progressively reconstruct the edges of the optimum $x^*$ in at least $2+L_{n-7}+2+L_{n-9}+\ldots$  steps.\\

\begin{figure}

\centering
\includegraphics[width=13cm]{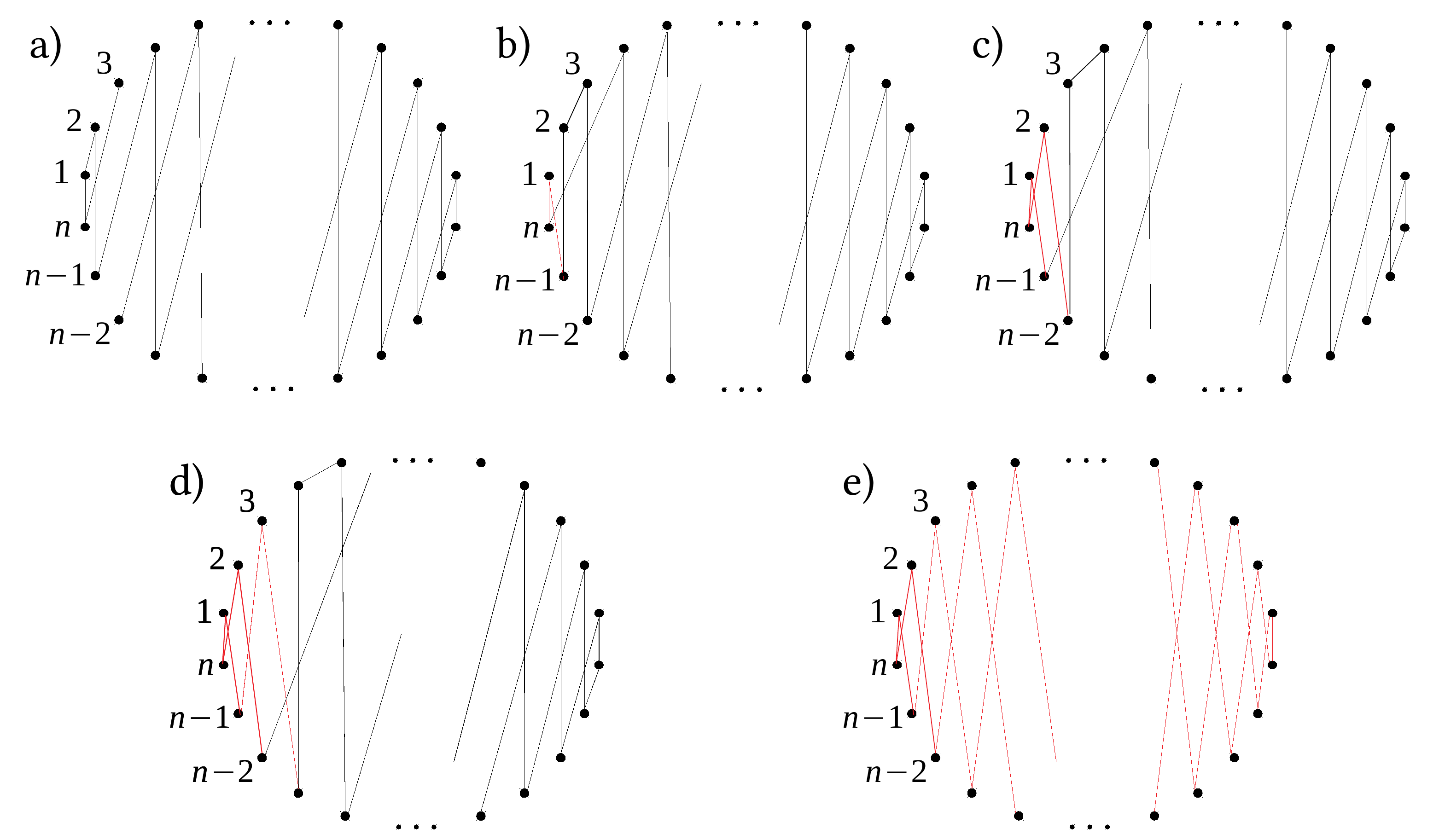}
\caption{Main steps of the monotone path. Once the red edges that belong to the optimum $e)$ are inserted they will never be deleted.}
\label{long_path_TSP}
\end{figure}

Together, we have $L_n \geq L_{n-1}+2+L_{n-3}+2+L_{n-5}+\ldots + 2+L_{k}$ with $k=4$ if $n$ is odd and $k=5$ otherwise. Define $\tilde L_n$ by $\tilde L_4=1$, $\tilde L_5=3$ and $\tilde L_n = \tilde L_{n-1}+2+\tilde L_{n-3}+2+\tilde L_{n-5}+\ldots + 2+\tilde L_{k}$. Then $L_n\geq \tilde L_n$ because 
$L_4\geq1$ and $L_5\geq 3$. Furthermore, $\tilde L_n=\tilde L_{n-1}+\tilde L_{n-2}+2$ therefore note that $\tilde L_n+2 = F_n$ is the Fibonacci sequence. Then $L_n\geq\tilde L_n \geq \frac{\phi^n}{\sqrt 5}-3$.
\end{proof}

%\begin{theorem}
%There exist monotone paths of length $>\frac{C}{n^2}\sqrt[3]{n!}$ on the shortest path polytope on the complete graph $K_n$ for some universal constant $C>0$.
%\end{theorem}

\begin{proof}[Proof of Theorem \ref{long_path_shortestpathpolytope}]
Recall that the vertices of the shortest path polytope are the paths from node say $1$ to $n$ (all nodes of the path should be distinct) and that two paths from $1$ to $n$ are adjacent if and only if the union of their graphs forms a unique cycle. Denote by $x_{i,j}$ the coordinate of the edge going from node $i\neq n$ to $j\neq 1$ in a vertex $x$ of the polytope. Similarly to the cost function used in Theorem \ref{long_path_tsp} we use the linear function
\begin{equation*}
\psi = x_{1,2}+\alpha x_{1,3} + \ldots + \alpha^{n-2} x_{1,n}+\alpha^{n-1} x_{2,3}+\ldots+\alpha^{2n-4}x_{2,n}+\ldots+\alpha^{n^2-3n+2}x_{n-1,n},
\end{equation*}
so that the linear order is the lexicographic order on the edges\\ $\{1,2\},\{1,3\},\ldots,\{1,n\},\{2,3\},\ldots,\{2,n\},\{3,2\},\{3,4\},\ldots,\{3,n\},\ldots,\{n-1,n\}$ with a chosen small enough $\alpha>0$. We start from the path $1,2,\ldots,n$ which is the maximum value vertex for $\psi$. Denote by $L_n$ the length of the monotone path we will construct here by induction.

Step 1: Fix the edge $x_{1,2}=1$. This facet corresponds to the shortest path polytope on the complete graph $K_{n-1}$ with nodes $2,3,\ldots,n$. The objective function $\psi$ is still the same lexicographic order on the edges of $K_{n-1}$. Then, by induction, we can get to path $1,2,n$ in $L_{n-1}$ monotone steps. 

Step 2: We now get to the path $1,3,4,\ldots,n$ which is a decreasing neighbor because we do not use the edge $\{1,2\}$ anymore. Similarly to Step 1, we get to path $1,3,n$ in $L_{n-2}$ monotone steps.

\begin{figure}
    \centering
    \includegraphics[width=15cm]{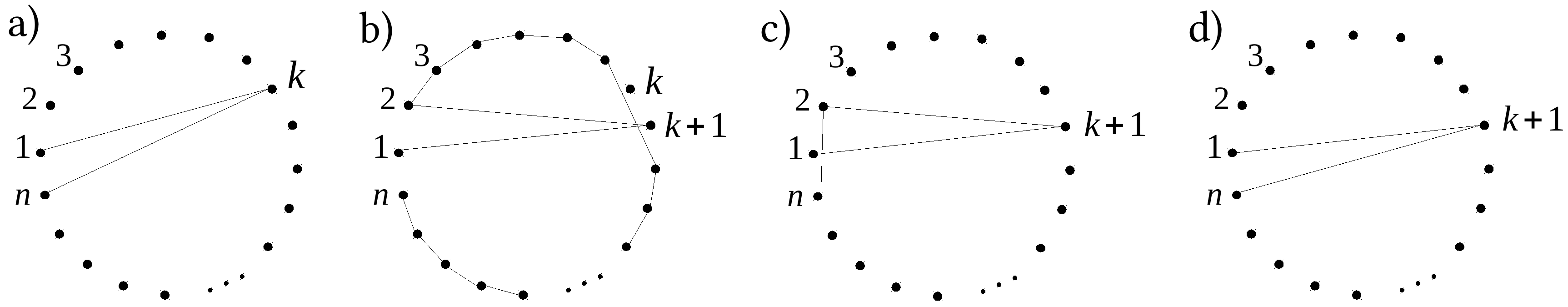}
    \caption{Step 3 of the long monotone path on the shortest path polytope. The length of the path from a) to d) is $L_{n-3}+2$.}
    \label{shortest_path_fig_step3}
\end{figure}

Step 3: We are now going to go from path $1,3,n$ to $1,4,n$, then to $1,5,n$ etc... to $1,n-1,n$ (see Figure \ref{shortest_path_fig_step3}). From the path $1,k,n$ where $k\geq 3$ we get to the decreasing neighbor $1,k+1,2,3,\ldots,k-1,k+2,k+3,\ldots,n$ (see Figure \ref{shortest_path_fig_step3} b)). Fixing edges $x_{1,k+1}=x_{k+1,2}=1$, this facet is equivalent to the shortest path on the complete graph $K_{n-3}$ with nodes $2,3,\ldots,k-1,k+2,k+3,\dots,n$, starting in $2$ and ending in $n$. We therefore get to path $1,k+1,2,n$ in $L_{n-3}$ steps and then to path $1,k+1,n$ in an improving step. We can repeat this operation $n-4$ times until we reach path $1,n-1,n$. We finally get to path $1,n$ in one improving step.
All together we get 
\begin{equation*}
L_n = L_{n-1}+L_{n-2}+(n-4)L_{n-3} + 2(n-3)\geq (n-2)L_{n-3}.
\end{equation*}
Therefore $L_{3k+2}\geq 3^k \cdot k!$ and $L_{3k+1},L_{3k}\geq 3^{k-1} \cdot (k-1)!$ where $3^k \cdot k!\sim \frac{\tilde C}{k^{1/3}} \sqrt[3]{(3k)!}$ for some constant $\tilde C$. The result follows.
\end{proof}

Although the height of all the combinatorial polytopes above is exponential, several authors have shown that their monotone diameter can be short. For example Rispoli \cite{rispoli1992} showed that the monotone diameter of the Birkhoff polytope of vertices in $\mathcal{S}_n$ is $\lfloor \frac{n}{2} \rfloor$. Furthermore, he also proved that several matching polytopes \cite{rispoli1994}, the shortest path polytope \cite{rispoli1992} and the TSP \cite{rispoli1998} have linear monotone diameter. 

We now give estimates for their pivot height for some specific pivot rules. For this we use an analysis of the number of basic feasible solutions (BFS) generated by the algorithm. The ideas we use are inspired from the work of Kitahara, Mizuno and co-authors (see \cite{kitahara2012}, \cite{KitaharaMizuno2013} and \cite{Tano}).\\

%Several authors, including Balas Padberg and Rispoli et al. \cite{} have studied the monotone diameter of several of combinatorial polytopes.  For instance the Birkhoff polytope of vertices in $\mathcal{S}_n$ is $\lfloor \frac{n}{2} \rfloor$.\\

Consider the following linear program in canonical form for a bounded polytope:

\begin{eqnarray} \label{eq:LP}
	\min && c^Tx \\
	s.t. && Ax=b,\quad x\geq 0 \nonumber
\end{eqnarray}
where $A\in\R^{m\times n}$, $m<n$ and $A$ is a matrix with full row rank.

For a given BFS $x$, let $B$ and $N$ denote the submatrices of $A$ corresponding to basic and non-basic columns respectively.  We split the objective function vector $c$ and the variables $x$ accordingly,

\begin{equation*}
	c=\begin{bmatrix}
		c_B \\ c_N
	\end{bmatrix},\quad x=\begin{bmatrix}
		x_B \\ x_N
	\end{bmatrix},\quad x_B=B^{-1}b,\quad x_N=0.
\end{equation*}

\begin{definition}
Define $\gamma$ and $\delta$ respectively as the maximum and the minimum among the positive coordinates of all BFS. 
We also denote by $\nu$ and $\mu$ respectively the maximum and minimum among the $\ell_2$-norms of all possible edges.
\end{definition}

In the paper \cite{KitaharaMizuno2013} Kitahara and Mizuno proved that for Dantzig's pivot rule and the greatest  
(descent) improvement pivot rule, the number of steps is bounded by $n\left[\frac{m\gamma}{\delta}\log\left(m\frac{\gamma}{\delta}\right)\right]$ 
iterations. In \cite{kitahara2012} Kitahara, Matsui and Mizuno improved that result and obtained the following bound: $(n-m)\left[\min\{m,n-m\}\frac{\gamma}{\delta}\log\left(\min \{m,n-m\}\frac{\gamma}{\delta}\right)\right]$. Here we provide a new bound for another very popular pivot rule, the \emph{steepest edge} pivot rule. 
We remark that this bound is in general still exponential in the bit-size of the
input, and that the constants are complicated to compute. For example, $\delta$ is NP-hard to compute in general (see \cite{Kuno+Sano2018}).

Consider now a single step of steepest edge pivoting rule for the Simplex method.   To simplify the argument, we assume that the current basis 
consists of the first $m$ columns. If column $q$ ($q>m$) is entering the basis and the column $p$ is leaving the basis, then the next BFS $\bar{x}$ we encounter would be of the form
$$
	\bar{x}=x+\theta\eta_q
$$
where $\theta$ is the length of the step, and $\eta_N^q$ is one from the set of edge directions $\eta_N=[\eta_N^{m+1},...,\eta_N^{n}]$,
$$
	\eta_N^q = \begin{bmatrix}
	-B^{-1}N\\ I
	\end{bmatrix} e_{q-m}.
$$
Let $\bar c_N$ denote the reduced cost vector for non-basic variables, so
\begin{equation}
	\bar c_N^q = c^T\eta_N^q, \quad \bar c_N^T=c^T\eta_N=c_N^T-c_B^TB^{-1}N.  %z is also a column
\end{equation}
Denote by $\zeta_N^q$ the Euclidean norm of $q-$th edge, and $W_N$ a diagonal matrix whose diagonal elements are $\zeta_N^q$.
$$
\zeta_N^q=\|\eta_N^q\|_2,\quad W_N=diag(\zeta_N^{m+1},...,\zeta_N^n).
$$ 
In the steepest edge Simplex algorithm, we determine our pivoting column by minimizing the normalized reduced cost i.e., choosing $\hat{q}$ such that
$$
\hat{q}=\arg\min \bar c_N^q/\zeta_N^q.
$$
Set $\Lambda=-\bar c_N^{\hat{q}}/\zeta_N^{\hat{q}}>0$. With all the notations above, Problem \ref{eq:LP} can be rewritten as
\begin{eqnarray} \label{eq:LP2}
	\min_{x_N} && c_B^TB^{-1}b+\bar c_N^TW_N^{-T}W_Nx_N. \\
	s.t. && x_B=B^{-1}b-B^{-1}Nx_n, \nonumber\\
	&&x_B\geq 0,\quad x_N\geq 0. \nonumber
\end{eqnarray}
Note that $W_N^{-1} \bar c_N$ is the normalized reduced cost vector.\\

%Here we make the final assumption that the $l_2$ norms of all possible edge directions are bounded below by $\mu$ and above by $\nu$. 
%And the values of all the positive elements of BFS are bounded below by $\delta$ and above by $\gamma$. Later on we will give explicit 
%bounds on the ratios $\nu/\mu$.

Lemma 1 of \cite{KitaharaMizuno2013} gives an upper bound on the distance between the current objective value and the optimal value. The following lemma is an extension for the steepest edge pivoting rule. 

\begin{lemma}
    \label{inequality_delta}
Assume $z^*$ is the optimal value and $x^{(t)}$ the BFS generated at the $t-th$ iteration, with the corresponding basic and non-basic columns $B^{(t)}, N^{(t)}$. Then we have
	\begin{eqnarray*}
	z^*\geq c^Tx^{(t)}-\Lambda^{(t)} m\nu \gamma.
	\end{eqnarray*}
\end{lemma}
\begin{proof}
We decompose the optimal value $z^*$ with the current basis.
	\begin{eqnarray*}
		z^*&=&c^Tx^*\\
			&=&c^Tx^{(t)}+\bar c_{ N^{(t)}}^Tx_{ N^{(t)}}^* \\ %% use the magic representation
			 & = & c^Tx^{(t)}+\bar c_{ N^{(t)}}^TW_{  N^{(t)}}^{-T} W_{N^{(t)}}x_{ N^{(t)}}^*.
	\end{eqnarray*}
Using the definition of $\Lambda^{(t)}$	we get
	\begin{eqnarray*}
		z^*&\geq & c^Tx^{(t)}-\Lambda^{(t)} e^T W_{ N^{(t)}}x_{ N^{(t)}}^*\\
		& \geq & c^Tx^{(t)}-\Lambda^{(t)} \left(e^T W_{ N^{(t)}} e\right) \gamma\\
			 & \geq & c^Tx^{(t)}-\Lambda^{(t)} m\nu\gamma,
	\end{eqnarray*}
	where the last inequality results from the definition of $\nu$.
\end{proof}

The following theorem shows the decreasing rate of the gap between the optimal value and the objective value at iteration $t$.
\begin{theorem}
	\label{steepest_edge_algo}
For the steepest edge pivoting rule, if the $t$-th iterate $x^{(t)}$ is not optimal then
\begin{eqnarray*}
	\frac{c^Tx^{(t+1)}-z^* }{c^Tx^{(t)}-z^*}  \leq 1-\frac{\mu\delta}{m\nu\gamma}.
	\end{eqnarray*}
\end{theorem}
\begin{proof}
	\begin{eqnarray*}
		c^Tx^{(t)}-c^Tx^{(t+1)}&=&\Lambda^{(t)}\zeta_{N^{(t)}}^{\hat{q}^{(t)}}x_{\hat{q}^{(t)}}^{(t+1)} \\
		&\geq & \Lambda^{(t)} \mu\delta\\
		& \geq & \frac{\mu\delta}{m\nu\gamma} (c^Tx^{(t)}-z^*).
	\end{eqnarray*}
	The last inequality follows from Lemma \ref{inequality_delta}. Rearranging the terms gives us the desired result.
\end{proof}
Lemma 2 in the original paper \cite{KitaharaMizuno2013} does not depend on pivoting rules, so it can be applied directly here.
\begin{lemma}(Kitahara and Mizuno, Lemma 2 in \cite{KitaharaMizuno2013}) 
\label{lemma2kitahara}
	If $x^{(t)}$ is not optimal, then there exists $\bar{j}\in B^t$, such that $x^{(t)}>0$, and for any $k$, $x^{(k)}$ satisfies
 	$$
 		x_{\bar{j}}^{(k)}\leq \frac{m(c^Tx^{(k)}-z^*) }{c^Tx^{(t)}-z^*} x_{\bar{j}}^{(t)}.
	$$
\end{lemma}
Combining the results from Theorem \ref{steepest_edge_algo} and Lemma \ref{lemma2kitahara}, we have the following lemma.
\begin{lemma}
\label{index_exponential_decrease}
	If $x^{(t)}$ is not an optimal solution, then there exists $\bar{j}\in B^t$, such that $x_{\bar{j}}^{(t)}>0$ and becomes zero and stays zero after $\left[\frac{m\gamma\nu}{\delta\mu}log\left(m\frac{\gamma}{\delta}\right)\right]$ iterations.
\end{lemma}

\begin{proof}
	\begin{eqnarray*}
	x_{\bar{j}}^{(t+k)}\leq m\left(1-\frac{\mu\delta}{m\nu\gamma}\right)^k x_{\bar{j}}^{(t)} \leq m\gamma\left(1-\frac{\mu\delta}{m\nu\gamma}\right)^k \leq m\gamma \exp\left(-\frac{k\mu\delta}{m\nu\gamma}\right).
	\end{eqnarray*}
	Therefore, if $k>\left[\frac{m\gamma\nu}{\delta\mu}log \left(m\frac{\gamma}{\delta}\right)\right]$, we would have $x_{\bar{j}}^{(t+k)}<\delta$. By the definition of $\delta$, the lemma follows.
\end{proof}
The event described in Lemma \ref{index_exponential_decrease} can happen at most once for each variable. Since we have in total $n$ variables, we have the following theorem.
\begin{theorem}
\label{theorem_steepest_edge}
Using steepest edge algorithm to solve Problem (\ref{eq:LP}) would generate at most 
\begin{equation}\label{estimation}
n\left[\frac{m\gamma\nu}{\delta\mu} \log\left(m\frac{\gamma}{\delta}\right)\right]
\end{equation}
 different BFS. In other words, the algorithm would reach the optimal solution in at most $n\left[\frac{m\gamma\nu}{\delta\mu}\log\left(m\frac{\gamma}{\delta}\right)\right]$ non-degenerate pivots.
\end{theorem}

Here is another bound in terms of the sub-determinants of the input matrix $A$. In the following, we will denote by $\Delta$ and $\lambda$ respectively the maximum and minimum absolute value of non-zero determinants over the $m\times m$ sub-matrices of $A$.

\begin{lemma}
\label{subdeterminant_bound}
For any $m\times m$ sub-matrix $B$ of $A$ and any column $A_k$ of the matrix $A$, $\| B^{-1}A_k \|_2 \leq \sqrt{m}\frac{\Delta}{\lambda}$.
\end{lemma}

\begin{proof} By Cramer's rule, the $j$-th entry of $B^{-1}A_k$ is given by $\frac{\det(B_j)}{\det(B)}$ for any $j \in \{1, \ldots, m\}$, where $B_j$ is the matrix obtained by replacing the $j$-th column of $B$ by $A_k$. Since $A_k$ is also a column of $A$, $B_j$ is an $m\times m$ submatrix of $A$. Thus, $|\frac{\det(B_j)}{\det(B)}| \leq \frac{\Delta}{\lambda}$. The bound follows.
\end{proof}

\begin{cor}
\label{steepest_edge_lambda}
Using steepest edge algorithm to solve Problem (\ref{eq:LP}) would generate at most
\begin{equation}
    n\left[ m\sqrt{2m}\frac{\gamma^2\Delta}{\delta^2\lambda}\log\left(m\frac{\gamma}{\delta}\right)\right]
\end{equation}
different BFS.
\end{cor}

\begin{proof}
Let $x_B$ be the vertex corresponding to a basis $B$, and a neighbor $\tilde x$. Denote by $\hat q$ the entering variable to get from $x_B$ to $\tilde x$. Then $\tilde x-x_B = -\tilde x_{\hat q}A_B^{-1} A_{\hat q}$ where $A_{\hat q}$ is the $\hat q$-th column of $A$ and $A_B$ is the $m\times m$ submatrix of $A$ of columns in the basis $B$. Then,
\begin{equation*}
    \|\tilde x - x_B\|_2=\tilde x_{\hat q} \sqrt{1+\|A_B^{-1}A_{\hat q}\|_2^2}.
\end{equation*}
By Lemma \ref{subdeterminant_bound}, $\nu\leq \gamma \sqrt{1+m\left(\frac{\Delta}{\lambda}\right)^2}$ and $\mu\geq \delta$. The proof follows from the bound given in Theorem \ref{theorem_steepest_edge}.
\end{proof}

When the matrix $A$ is totally unimodular, Corollary \ref{steepest_edge_lambda} gives a bound for the number of different BFS of $
    n\left[ m\sqrt{2m}\frac{\gamma^2}{\delta^2}\log\left(m\frac{\gamma}{\delta}\right)\right]$
for the steepest edge rule. Remark that in this case we get a very similar bound to that given by Tano, Miyashiro and Kitahara \cite{Tano}. In their paper they show that the number of different BFS with the generalized $p$-norm steepest edge rule is bounded by
\begin{equation*}
    (n-m)\left[ m^{1+1/p}\frac{\gamma^2}{\delta^2}\log\left(m\frac{\gamma}{\delta}\right)\right].
\end{equation*}
In addition, when $b$ is integral, Kitahara and Mizuno \cite{KitaharaMizuno2013} derived from their result the bound $n[m \|b\|_1\log(m \|b\|_1)]$ on the number of different BFS generated by the simplex method with Dantzig's rule or the greatest improvement rule. Here we improve this result for different polytopes of interest and give the corresponding explicit polynomial bounds.

\begin{cor}
\label{bound_tp}
	Using Dantzig's pivot rule or greatest improvement pivot rule to solve a transportation problem written as $Ax=b$, $x\geq 0$ generates at most
	
	\begin{equation}
	n\left[ \|b\|_1 \log\left(m \|b\|_\infty\right)\right]
	\end{equation}
	different BFS and more precisely at most $n[ S \log(m \|b\|_\infty)]$ different BFS where $S$ is the total supply, equal to the total demand in the transportation problem.
\end{cor}

\begin{proof}
	We slightly change the proof of the result given by Kitahara and Mizuno \cite{KitaharaMizuno2013}.
	\begin{eqnarray*}
		z^*&=&c^Tx^*\\
		&=&c^Tx^{(t)}+\bar c_{ N^{(t)}}^Tx_{ N^{(t)}}^* \\ %% use the magic representation
		& \geq & c^Tx^{(t)}-\Delta^{(t)} \|x_{ N^{(t)}}^*\|_1
	\end{eqnarray*}
	where $\Delta^{(t)} = - \min \bar c_N^q$. If $x_{i,j}$ is the value for the edge from supply node $i$ to demand node $j$, $\|x_{ N^{(t)}}^*\|_1\leq \|x^*\|_1 \leq \sum_{i,j} x^*_{i,j}=S$ the total supply (or total demand).
	Similarly to the proof of Theorem \ref{steepest_edge_algo}, we use the above inequality to find
	\begin{eqnarray*}
		c^Tx^{(t)}-c^Tx^{(t+1)}&=&\Delta^{(t)} x_{\hat{q}^{(t)}}^{(t+1)} \\
		&\geq & \Delta^{(t)}\delta\\
		& \geq & \frac{\delta}{S} (c^Tx^{(t)}-z^*).
	\end{eqnarray*}
	
Therefore $c^T x^{(t+1)}-z^* \leq \left( 1-\frac{\delta}{S}\right) (c^T x^{(t)}-z^*)$. Using Lemma \ref{lemma2kitahara}, we get 
\begin{equation*}
x_{\bar{j}}^{(t+k)}\leq m\left(1-\frac{\delta}{S}\right)^k x_{\bar{j}}^{(t)} \leq m\gamma\left(1-\frac{\delta}{S}\right)^k \leq m\gamma e^{-\frac{k\delta}{S}}.
\end{equation*}
The number of different BFS is then at most $n[\frac{S}{\delta} \log(m\frac{\gamma}{\delta})]$. As noted in \cite{KitaharaMizuno2013}, since $A$ is a totally unimodular matrix, $\delta$ is a positive integer, so $\delta\geq 1$. Denote by $s_i$ and $v_j$ the supply and demand at supply node $i$ and demand node $j$ respectively. Then $\gamma = \max x_{i,j} \leq \min (\max_i s_i,\max_j d_j) \leq \|b\|_\infty$. The proof follows.
\end{proof}

Note that the proof of Corollary \ref{bound_tp} gives the bound $n\left[\frac{\|x^*\|_1}{\delta} \log\left(m\frac{\gamma}{\delta}\right)\right]$ on the number of different BFS generated for Dantzig and greatest improvement pivot rules and a similar bound for the steepest edge pivot rule: $n\left[\frac{\|x^*\|_1 \nu}{\delta\mu} \log\left(m\frac{\gamma}{\delta}\right)\right]$. We are now ready to use the above results to prove our bounds on several combinatorial polytopes.

\begin{proof}[Proof of Theorem \ref{pivot_diameter_bland} and Theorem \ref{pivot_diameter_steepest}]
We prove the two theorems in parallel, as we only need to
apply two different estimations to the same polytope for each item of the same index.
\begin{enumerate}[a)]
\item %The fractional perfect matching polytope on the graph $(V,E)$ is defined by the following equations:

%\begin{equation*}
%FPM = \{ x(\delta(v))=1 \; \forall v\in V,\; x_{e}\geq 0 \; \forall e \in E \}.
%\end{equation*}
The fractional perfect matching polytope is a $0/\frac{1}{2}/1$ polytope so $\gamma=1$ and $\delta=1/2$. Furthermore, $x\in FPM$ is a vertex if and only if it is the union of a perfect matching $\mathcal{M}_x$ given by the edges $\{e\in E,x_e=1\}$ and a collection $\mathcal{C}_x$ of disjoint cycles of odd length given by the edges $\{e\in E,x_e=1/2\}$. Then $\|x\|_1=\frac{k_1+k_2}{2}$ where $k_1$ is the number of nodes in the odd length cycles and $k_2$ the number of nodes in the matching $\mathcal{M}_x$. Therefore $\|x^*\|_1= \frac{|V|}{2}$. Now let us give bounds for $\mu$ and $\nu$. For two vertices $x_1$ and $x_2$ and any edge $e\in E$, $|(x_1-x_2)_e|\leq 1$. Then, $\|x_1-x_2\|_2^2 \leq \|x_1-x_2\|_1\leq |V|$ so $\nu\leq \sqrt{|V|}$. Furthermore $\mu \geq \delta=1/2$.

\item %The fractional matching polytope on the graph $(V,E)$ is defined by:
%\begin{equation*}
%FM = \{ x(\delta(v))\leq 1 \; \forall v\in V,\; x_{e}\geq 0 \; \forall e \in E \}.
%\end{equation*}

%Note that $FPM$ is a facet of the $FM$.
The fractional matching polytope is still a half integral polytope so $\gamma=1$ and $\delta=1/2$. Vertices are still the union of a perfect matching on $\mathcal{M}_x$ given by the edges $\{e\in E,x_e=1\}$ and disjoint odd-length cycles $\mathcal{C}_x$ given by the edges $\{e\in E,x_e=1/2\}$. We have to add the $n$ slack variables $s_i$ for the inequality at each node so $\|x\|_1=|\mathcal{M}_x| /2+|\mathcal{C}_x|/2+|V-(\mathcal{M}_x \cup \mathcal{C}_x)|$ where the last term comes from the slack variables. Then, $\|x^*\|_1\leq |V|$. The same arguments as above give $\mu\geq 1/ 2$ and $\nu\leq \sqrt{2|V|}$.

The next polytopes are $0/1$ polytopes, therefore $\gamma=\delta=1$.

\item The Birkhoff polytope has exactly $n$ positive edges then $\|x\|_1=n$ for any permutation $x$. Two vertices $x,y$ are adjacent on this polytope if the symmetric difference of their edges form a single alternating cycle of norm $\sqrt{l}$ where $l$ is its length. Because the cycle is alternating, we have $4\leq l\leq 2n$ and then $\mu = 2$, $\nu=\sqrt{2n}$.

\item %The shortest path polytope can be defined as the convex hull of the paths from node $1$ to $n$ without cycles. A system of equations is
%\begin{equation*}
%\{ (x_{i,j}\geq 0)_{1\leq i\leq n-1,2\leq j\leq n}\, s.t. \, \sum_{j=2}^n x_{1,j} =1, \sum_{j\neq i} x_{i,j} -\sum_{j\neq i} x_{j,i}=0, \sum_{j\neq i}x_{i,j}\leq 1 ,\,2\leq i\leq n-1\}.
%\end{equation*}
For the shortest path polytope, there are $n^2-3n+3$ variables and $n-2$ slack variables for each node of indices $2$ to $n$. A path of length $l$ is represented by a vertex $x$ where the positive slack variables are the variables for the nodes which are not visited by the path. Then $\|x\|_1 = l+ (n-1-l) = n-1$. Two paths are adjacent if the union of their edges contains a unique cycle. The norm of the corresponding direction is at least $\sqrt {l'}$ where $l'$ is the length of this cycle and at most $\sqrt{2l'}$ where we consider the $l'$ possibly affected slack variables. Therefore $\mu\geq\sqrt 3$ and $\nu\leq\sqrt {2n}$.
\end{enumerate}
\end{proof}

\section{Monotone paths on Transportation polytopes}
\label{transportation_polytopes}

Exponentially-long simplex paths can be found even for very simple linear programs given by network flow problems using
Dantzig's pivot rule \cite{zadeh}. Nevertheless, Orlin showed that for certain pivot rules, the network Simplex method runs in a polynomial number of pivots \cite{orlin97}. Here we try to look at the case of transportation polytopes. 

In the paper \cite{borgwardt2018}, Borgwardt, De Loera and Finhold proved that the diameter of $m\times n$ transportation polytopes is bounded by the Hirsch bound $m+n-1$. In this section we study the monotone diameter of this polytope. From any degenerate transportation we can derive a non-degenerate transportation polytope with greater or equal monotone diameter by perturbing the original polytope. We will therefore assume non-degeneracy in this section. Recall that for a non-degenerate transportation polytope $P$, $x\in P$ is a vertex if and only if its support forms a spanning tree on the bipartite graph $K_{m,n}$ given by the $m$ supply nodes and the $n$ demand nodes (see references in \cite{borgwardt2018}). For a vertex $x$ we will write $s\sim d$ when supply node $s$ and demand node $d$ are adjacent in the support graph of $x$.

\begin{lemma}
\label{type_edges}
Let $x^*$ be the optimum of a $n\times m$ transportation polytope for a given linear functional $c$. Denote by $c_{v,w}$ the cost of the edge between vertex $v$ and $w$. Let $s_1,s_2,\ldots,s_k$ be $k\geq 2$ supply nodes and $d_1,d_2,\ldots,d_k$ demand nodes. If $s_1\sim d_1,s_2\sim d_2,\ldots, s_k\sim d_k$ in $x^*$ then $c_{s_1,d_1} - c_{d_1,s_2} + c_{s_2,d_2} - c_{d_2,s_3} +\ldots+c_{s_k,d_k} - c_{d_k,s_1} < 0$.

Therefore, an edge between two vertices of the transportation polytope following the cycle $s_1 d_1 s_2 d_2 \ldots s_k d_k$ is an improving edge for the linear functional.
\end{lemma}

\begin{proof}
Let $s$ and $d$ be respectively a supply and demand node which are not adjacent in $x^*$. Let $s=x^0,x^1,x^2,\ldots,x^l=d$ be the path from $s$ to $d$ in $x^*$. By optimality of $x^*$, entering the edge $(s,d)$ into the spanning tree associated to $x^*$ will increase the cost function. In other words, the reduced cost of the variable $(s,d)$ is positive i.e., $\tilde C_{s,d}:= c_{s,d}-c_{x^0,x^1}+c_{x^1 x^2}-\ldots+c_{x^{l-2},x^{l-1}} -c_{x^{l-1},x^l} >0$, which gives us an inequality on the alternating cycle $s=x^0, x^1, x^2, \ldots, x^l=d$.

We will add $k$ inequalities of this type to obtain the desired inequality. More precisely, we will add the inequality resulting from the cycle given by adding the edge $(s_2,d_1)$ to $x^*$, the cycle given by the edge $(s_3,d_2)$, etc... and the cycle given by $(s_1,d_k)$. We prove by induction on $k$ that in the resulting sum $\tilde C_{s_2,d_1}+\tilde C_{s_3, d_2}+\ldots+\tilde C_{s_1,d_k}$, terms cancel out to leave out $-(c_{s_1,d_1} - c_{d_1,s_2} + c_{s_2,d_2} - c_{d_2,s_3} +\ldots+c_{s_k,d_k} - c_{d_k,s_1})$, which will then be positive.

Denote by $T$ the smallest subtree of the support spanning tree of $x^*$ containing the edges $(s_1,d_1),(s_2,d_2),\ldots,(s_k,d_k)$. Without loss of generality, assume $(s_1,d_1)$ is a leaf in $T$. We are going to merge together $\tilde C_{s_2, d_1}$ and $\tilde C_{s_1, d_k}$. The term $-c_{s_1,d_1}$ appears exactly once in their sum, say in $\tilde C_{s_1, d_k}$. We can therefore write the two paths in $x^*$ from $d_1$ to $s_2$ and $s_1$ to $d_k$ by $d_1 v^1 v^2\ldots v^l p^1 p^2 \ldots p^{r-1} p^r=s_2$ and $s_1 d_1 v^1 v^2 \ldots v^l q^1 q^2 \ldots q^{t-1} q^t=d_k$ where $p^1\neq q^1$. Note that the path in $x^*$ from $d_k$ to $s_2$ is exactly $q^t q^{t-1} \ldots q^1 v^l p^1 p^2\ldots p^r$. Then the terms from the path $d_1v^1v^2\ldots v^l$ cancel to give $\tilde C_{s_1,q^t}+\tilde C_{d_1,p^r}=c_{s_2,d_1}+c_{s_1,d_k} - c_{s_1,d_1} - c_{s_2,d_k} +\tilde C_{s_2, d_k}$. 

If $k=2$, the above calculations directly give the desired result $\tilde C_{s_2\sim d_1}+\tilde C_{s_1\sim d_2} = c_{s_2,d_1}+c_{s_1,d_2} - c_{s_1,d_1} - c_{s_2,d_2}$. Otherwise, we use the induction on $\tilde C_{s_3\sim d_2}+\tilde C_{s_4\sim d_3}+\ldots+\tilde C_{s_2\sim d_k}$ and the result follows.

\end{proof}

%\begin{theorem}
%A $2 \times n$ transportation polytope has monotone diameter $\leq n$. Therefore, $2\times n$ transportation problems satisfy the monotone Hirsch conjecture.
%\end{theorem}

We now consider the case of a $2 \times n$ transportation polytope. We denote the supply and demand nodes respectively by $s_1,s_2$ and $d_1,\ldots,d_n$. Consider a vertex of the $2\times n$ transportation polytope. Assuming that the transportation polytope is non-degenerate, we can partition the demand nodes in the following way:
\begin{itemize}
\item the set $D_1$ of demand nodes that are leaves adjacent to supply node $s_1$ only.
\item the set $D_2$ of demand nodes that are leaves adjacent to supply node $s_2$ only.
\item the last demand node adjacent to $s_1$ and $s_2$.
\end{itemize}

\begin{proof}[Proof of Theorem \ref{monodiam_tp_polytope}]
We will show that from any vertex we can get to the optimum $x^*$ in at most $n$ steps using only edges of the type given by Lemma  \ref{type_edges}.

Without loss of generality, assume $d_1$ is adjacent to the two supply nodes in $x^*$, $D_1=\{2,\ldots,k\}$ and $D_1=\{k+1,\ldots,n\}$. We work by induction on $n\geq 1$. The result is true for $n=1$ and the monotone diameter is even $0=n-1$ so now assume $n>1$.
Let $x$ be the initial vertex of the transportation polytope. If any node $d\in D_1$ is a leaf incident to $s_1$ in $x$, likewise in $x^*$, we may remove this node and set the supply of $s_1$ to $S-D$ where $S$ and $D$ are respectively the supply at $s_1$, and the demand at $d$. The new problem is non-degenerate with $n-1$ demand nodes so the induction gives the desired result. The result similarly holds if a node in $D_2$ is a leaf adjacent to supply node $2$.

We therefore assume that all nodes in $D_1$ are adjacent to supply node $2$ and all nodes in $D_1$ are adjacent to supply node $1$ in $x$. Let $d$ the demand node adjacent to both supply nodes in $x$.\\

Case 1: $d\neq d_1$\\
We are in fact going to prove that only $n-1$ steps are necessary to get to the optimum.

If $D_1$ and $D_2$ are not empty (see Figure \ref{twotimesn}b)), without loss of generality, assume $d\in D_1$ and let $\tilde d \in D_2$. We make the edge $(s_2,\tilde d)$ enter the basis. The corresponding cycle in $x$ is $s_2 d s_1 \tilde d$ with $(s_2 \tilde d)$ and $(s_1,d)$ being two edges present in the optimum $x^*$. By Lemma \ref{type_edges}, this pivot reduces the cost function. Denote by $x^2$ the resulting vertex. The demand node of the edge which has been deleted, either $(s_2,d)$ or $(s_1,\tilde d)$ is now a leaf in $x^2$ adjacent to the same supply node as in $x^*$. Similarly to above, we can delete this demand node and we get the result by induction.

\begin{figure}

\centering
\includegraphics[height=4cm]{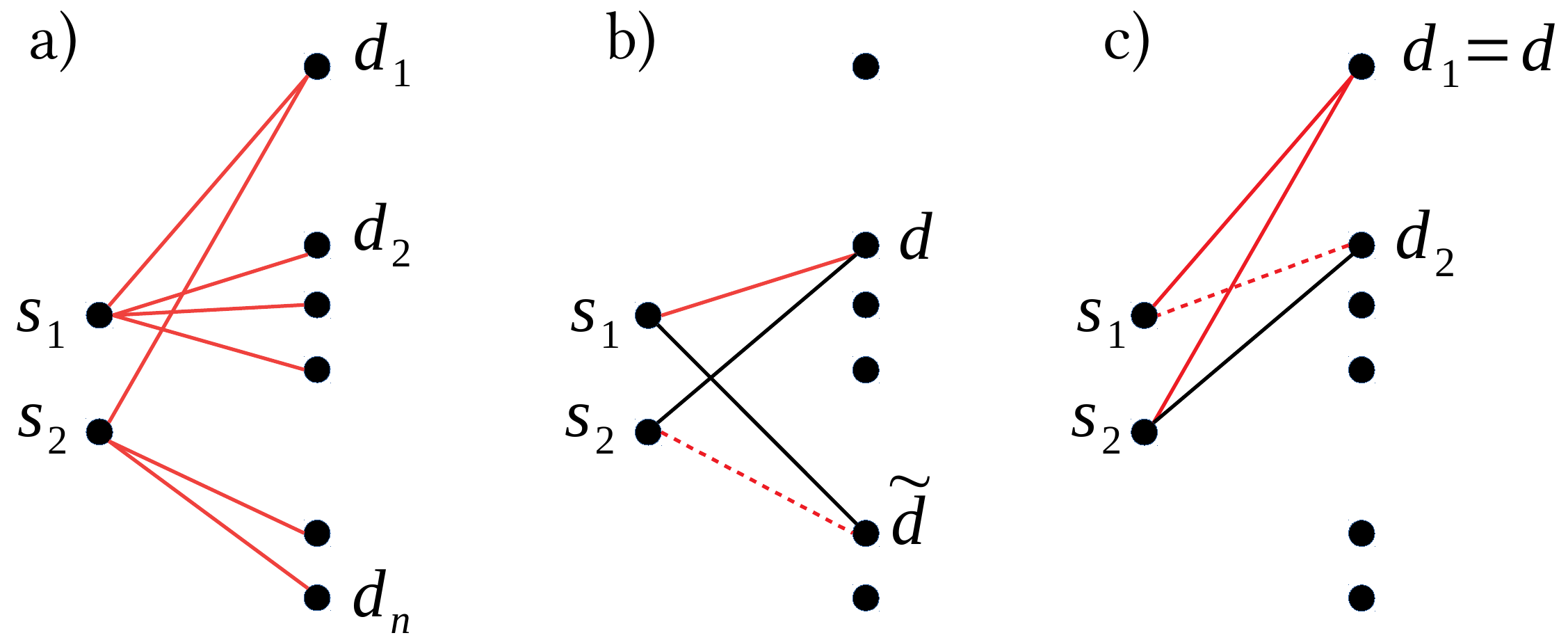}
\caption{Illustration of the choice of entering variable in dashed lines when $D_1$ and $D_2$ non empty. Edges belonging to the optimum tree a) are in red.}
\label{twotimesn}
\end{figure}

Otherwise, without loss of generality we can assume $D_2$ empty and $D_1=\{2,\ldots,n\}$ (see Figure \ref{D2null}). $s_2$ is a leaf adjacent to $d_1$ in $x^*$ so the demand at $d_1$ is greater to the supply at $s_2$. Then, in an admissible tree, $d_1$ cannot be a leaf adjacent to $s_2$.
Since $d\neq d_1$, $d_1$ is a leaf and it has to be adjacent to $s_1$ in $x$. We make the variable $(s_2,d_1)$ enter the basis. The corresponding cycle is $s_2 d s_1 d_1$ and $(s_1,d)$ and $(s_2,d)$ are present edges in the optimum $x^*$. By Lemma \ref{type_edges} this pivot is increasing. Denote by $x^2$ the new spanning tree. The potential leaving variables are only $(s_1,d_1)$ and $(s_2,d)$ but it cannot be $(s_1,d_1)$ otherwise $d_1$ would be a leaf adjacent to $s_2$ in $x^2$. Therefore, $(s_2,d)$ has been deleted and $d$ is now a leaf adjacent to the correct supply node in $x^2$ so we can delete the demand leaf $d$.

In $x^2$, $d_1$ is now adjacent to both supply nodes and all other demand nodes are adjacent to $s_2$. We enter the variable $(s_1,d_2)$ into the basis. The corresponding cycle $s_1 d_1 s_2 d_2$ is improving since $(s_1,d_2)$ and $(s_2,d_1)$ are in $x^*$. Similarly to above, $(s_1,d_1)$ cannot be the leaving variable otherwise $d_1$ would become a leaf adjacent to $s_2$. Therefore, in the new spanning tree $x^2$, $d_2$ is a leaf adjacent to the correct supply node so we can delete it. 

\begin{figure}

\centering
\includegraphics[height=3cm]{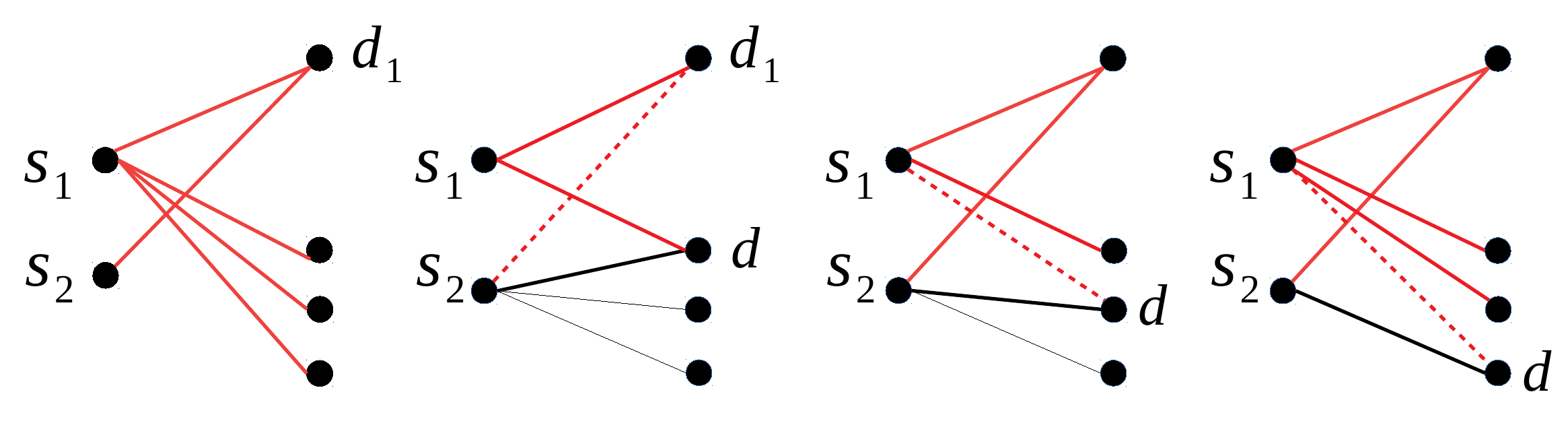}
\caption{Illustration of the choice of entering variable in dashed lines when $D_2$ null. Edges belonging to the optimum tree on the left are in red.}
\label{D2null}
\end{figure}

Note that in all pivot steps considered here we delete a demand node. In the new spanning tree, either $d_1$ is a leaf or $D_1$ or $D_2$ are null which are the cases we handled. The induction therefore holds and we can get to $n'=1$ in at most $n-1$ steps. For $n'=1$ there is only one spanning tree which is the optimum.

Case 2: $d=d_1$\\
We have already considered the case where $D_1$ or $D_2$ are empty so now assume this is not the case. Therefore $d_2\in D_1$ and $d_2$ is a leaf adjacent to $s_2$ in $x$ (see Figure \ref{twotimesn}c).

We make the edge $(s_1,d_2)$ enter the basis. The corresponding cycle is $s_1 d_1 s_2 d_2$. This is an improving cycle according to Lemma \ref{type_edges} given that edges $(s_1,d_2)$ and $(s_2,d_1)$  are present in $x^*$. Denote by $x^2$ the new vertex of the polytope. Either edge $(s_1, d_1)$ or $(s_2,d_2)$ has been removed. If $(s_2,d_2)$ was removed, $d_2$ is a leaf in $x^2$ adjacent to $s$ in $x^2$, likewise in $x^*$. Removing node $d_2$ therefore gives the result by induction. Otherwise, $(s_1, d_1)$ has been removed so in $x^2$, the demand node adjacent to both supply nodes is now $d_2\neq d_1$ and we use case 1.

We proved that the monotone diameter is $\leq n$. The bound $n$ can be attained potentially if there exists at least one vertex with $d=d_1$ and $D_1,D_2$ non empty. This can only happen if $n\geq 3$, otherwise the monotone diameter is $n-1$.
\end{proof}

\begin{conj}
The monotone diameter of $m\times n$  transportation polytopes is linear in $m$ and $n$.
\end{conj}

\noindent {\bf Acknowledgements:} The authors are grateful to Amitabh Basu, Laurent Poirrier, Francisco Santos, Laura Sanità, Sean Kafer and Xiaotie Chen for several remarks and comments that were useful to the creation of this manuscript. The first and second author were partially supported
by NSF grant DMS-1818969. The first author was also supported by \'Ecole Polytechnique.

%%%%%%%%%%%%%% BIB %%%%%%%%%%%%%%%%%%
%%%%%%%%%%%%%%%%%%%%%%%%%%%%%%%%%%%%%%
%%%%%%%%%%%%%%%%%%%%%%%%%%%%%%%%%%%%%%
\bibliographystyle{plain}
\bibliography{full-bib}
%%%%%%%%%%%%%% APPENDIX %%%%%%%%%%%%%%%%%%
%%%%%%%%%%%%%%%%%%%%%%%%%%%%%%%%%%%%%%
%%%%%%%%%%%%%%%%%%%%%%%%%%%%%%%%%%%%%%

\end{document}